\documentclass[11pt]{article}

\usepackage[a4paper,margin=1in]{geometry}
\usepackage{amsmath,amssymb,amsthm,mathtools}
\usepackage{bm}
\usepackage{enumitem}
\usepackage[colorlinks=true,linkcolor=blue,citecolor=blue,urlcolor=blue]{hyperref}
\usepackage{natbib}
\usepackage{graphicx,booktabs,tabularx}
\usepackage{subcaption}
\newtheorem{theorem}{Theorem}[section]
\newtheorem{lemma}[theorem]{Lemma}
\newtheorem{remark}[theorem]{Remark}
\newtheorem{assumption}[theorem]{Assumption}
\newtheorem{definition}[theorem]{Definition}

\usepackage{appendix}

\title{Equilibrium Singular Dividend Control with
Regime Switching and Bankruptcy under
Non-Exponential Discounting}
\author{Xiaochen Chen \and Guohui Guan \and Zongxia Liang \and Xiaodong Luo}
\date{}

\begin{document}
\maketitle

\begin{abstract}
This paper studies an infinite-horizon equilibrium dividend problem for a firm whose surplus follows a diffusion with observable Markov regime switching and is absorbed at bankruptcy. Dividend payments are modeled as singular controls, and non-exponential discounting gives rise to time inconsistency. We formulate equilibrium control laws and establish a verification theorem for a general finite-state model with surplus-dependent coefficients. For a two-state model with constant positive drifts and volatilities and bi-exponential discounting, we prove the existence of a regime-dependent equilibrium barrier strategy for arbitrary positive switching intensities. The proof relies on uniform estimates for a system of coupled boundary-value problems and the Poincaré–Miranda theorem. We also verify that the resulting strategy satisfies the equilibrium variational inequalities. Numerical results show that greater dispersion in discount rates, with the initial average discount rate held fixed, raises both barriers. The effect of switching intensity can be nonmonotone and may reverse the ordering of the barriers.
\end{abstract}


\medskip


\section{Introduction}
\label{sec:introduction}

Dividend distribution is a classical problem in insurance and corporate finance, where a firm must balance immediate payments to shareholders against retaining capital for future operations and protection against adverse fluctuations. The classical dividend problem, originating with \citet{deFinetti1957}, aims to maximize expected discounted dividends before ruin. For diffusion models, foundational results were established by \citet{JeanblancShiryaev1995} and \citet{AsmussenTaksar1997}. When dividend payments are unrestricted in rate and incur no fixed transaction costs, they can be modeled as a nondecreasing singular control. In many classical models, the optimal strategy is characterized by a barrier: surplus is retained below a threshold, while any excess is paid out immediately.
\vskip 5pt
A firm's operating environment may change over time, affecting both the profitability and volatility of its cash flows. Markov regime-switching models provide a tractable framework for capturing such changes, following \citet{Hamilton1989}. In the context of dividend optimization, \citet{sotomayor2011classical} study classical and singular control problems under regime switching, while \citet{jiang2012optimal} characterize optimal dividend policies for a Markov-modulated diffusion. These studies show that optimal dividend barriers may depend on the prevailing regime. In particular, a switch to a regime with a lower barrier may trigger an immediate lump-sum dividend payment, even though the surplus itself does not jump at the switching time.
\vskip 5pt
Discounting also plays an important role in dividend decisions. While exponential discounting is widely used, a single discount rate may not adequately reflect heterogeneous shareholder preferences. \citet{zhao2014dividend} consider non-exponential discounting in dividend problems to account for differences in shareholder preferences, while \citet{ebert2020weighted} show that aggregating exponential discount functions with different rates leads to weighted discounting. For example, two groups of shareholders with different discount rates give rise to
\begin{equation}
    \beta(t)=\theta e^{-\delta_1 t}
    +(1-\theta)e^{-\delta_2 t},
    \qquad 0<\theta<1,\quad 0<\delta_1<\delta_2.
    \label{eq:intro_discount}
\end{equation}
When the same discount function is applied to future payments at each decision date, preferences are generally time inconsistent. In particular, the relative valuation of future dividends changes as time passes, and thus a policy that is optimal at one decision date may no longer be optimal when reconsidered later.
\vskip 5pt
Without commitment, time-inconsistent decision making can be formulated as a game among successive selves. The idea of consistent planning dates back to \citet{Strotz1955}, while \citet{HarrisLaibson2001} study equilibrium behavior under hyperbolic discounting. In continuous time, \citet{EkelandPirvu2008} study investment and consumption without commitment, and \citet{BjorkKhapkoMurgoci2021} provide a general framework for time-inconsistent stochastic control. An equilibrium policy requires that no current self can improve its objective, to first order, by deviating over a short time interval, given that future selves follow the prescribed policy. In a regime-switching setting, \citet{ChenGuanLiang2025} study portfolio selection with regime-dependent market coefficients and beliefs-dependent utilities, where time inconsistency arises from the aggregation of certainty equivalents across terminal regimes. They derive the corresponding extended HJB equations and establish a verification theorem for equilibrium investment strategies.
\vskip 5pt
Several studies have examined time-inconsistent dividend problems. \citet{zhao2014dividend} derive equilibrium dividend strategies under non-exponential discounting with bounded dividend rates, while \citet{LiLiZeng2016} consider a dual risk model. \citet{ChenLiZeng2014} and \citet{ZhuSiuYang2020} study dividend problems under stochastic quasi-hyperbolic discounting, including singular dividend strategies. \citet{ChristensenLindensjo2022} investigate precommitment and consistent planning in a dividend problem with moment constraints and impulse payments. These studies consider different sources of time inconsistency and forms of dividend control. In the case of unrestricted singular controls, the equilibrium formulation must also account for instantaneous payments and the subsequent evolution of the controlled surplus.
\vskip 5pt
A general framework for this purpose is developed by \citet{liang2024equilibria}. 
They formulate equilibrium singular control through control laws specified by waiting and action regions, with the corresponding controlled process constructed through a Skorokhod problem. A local deviation is followed by a regenerated control law starting from the state reached after the deviation. This formulation incorporates the responses of future selves into the policy, rather than fixing a single cumulative control process in advance. They establish verification results and study the existence of equilibria for a finite-horizon singular control problem. Their framework provides the basis for the equilibrium formulation in the present paper, while the Markov-modulated dividend model of \citet{jiang2012optimal} provides the regime-switching foundation.
\vskip 5pt
Despite these developments, establishing the existence of equilibrium singular dividend policies in an infinite-horizon regime-switching model remains challenging. Non-exponential discounting requires auxiliary payoff functions, while regime-dependent intervention boundaries lead to coupled continuation equations across different regions. Moreover, solving the free-boundary equations alone does not guarantee an equilibrium, because the resulting candidate must also satisfy the variational inequalities, together with the admissibility and transversality conditions on an infinite horizon. This paper studies an infinite-horizon equilibrium dividend problem with regime switching, non-exponential discounting, and bankruptcy. The firm's surplus follows a diffusion whose drift and volatility depend on an observable finite-state Markov chain. Dividends are represented by a nondecreasing cumulative process $\xi$, which may include lump-sum payments. For initial surplus $x$ and regime $i$, the decision maker evaluates a dividend policy through
\begin{equation}
    J(x,i;\xi)
    =\mathbb{E}_{x,i}\!\left[
        \int_{[0,\tau^{\xi}]}\beta(t)\,d\xi_t
      \right],
    \label{eq:intro_objective}
\end{equation}
where $\tau^{\xi}$ is the bankruptcy time and the surplus is absorbed at zero. In the general formulation, the drift and volatility may also depend on the surplus, and $\beta$ is a positive, continuously differentiable, nonincreasing discount function with $\beta(0)=1$. Adapting the control-law approach of \citet{liang2024equilibria}, we seek a stationary equilibrium that is consistent with the decisions of successive selves. After a short admissible deviation, the prescribed law is restarted from the resulting surplus and the observed regime. This formulation also allows a regime transition to trigger an immediate dividend payment when required by the action region of the new regime.
\vskip 5pt
The combination of singular intervention, regime switching, and non-exponential discounting presents several analytical difficulties.
The current equilibrium value alone does not describe how future payments are reweighted after a delay; thus a family of auxiliary payoff functions is needed. 
Regime transitions couple these functions across states, while different dividend barriers give rise to different continuation intervals. 
Moreover, in the bi-exponential case, the two discount components must be evaluated under the same dividend law.
It is their weighted sum, rather than each component separately, that satisfies the equilibrium high-contact conditions at the policy barriers. 
Solving the resulting boundary equations is therefore only part of the argument. 
The candidate must also satisfy the equilibrium variational inequalities throughout the state space, and its reflected surplus process must meet the admissibility and infinite-horizon integrability requirements.
\vskip 5pt
The main contributions of this paper are as follows.
\begin{enumerate}
    \item We establish a verification theorem for equilibrium singular dividend control in a general finite-state regime-switching model with surplus-dependent coefficients. 
    The theorem uses auxiliary delayed-payoff functions and a coupled system of equilibrium variational inequalities. 
    It supplies sufficient analytic and probabilistic conditions under which a candidate control law is an equilibrium in the presence of the absorbing bankruptcy boundary.

    \item For a two-state model with constant positive drifts and volatilities and the bi-exponential discount function in \eqref{eq:intro_discount}, we reduce the barrier construction to a system of four linear ordinary differential equations, coupled across regimes, together with two free-boundary conditions. 
    We first establish uniqueness of the component payoffs for fixed positive barriers and continuity of the boundary residuals, including at coincident barriers.
    Small- and large-barrier sign estimates then allow us to apply the Poincar\'e--Miranda theorem and obtain a positive barrier pair. 
    This construction permits arbitrary positive transition intensities and does not require a prescribed ordering of the two barriers.
    
    \item We verify that the resulting candidate satisfies the full equilibrium conditions. 
    In particular, the equilibrium value is concave in the surplus, its marginal value is strictly greater than one in each waiting region, and the generator inequality holds in each action region. 
    We also establish the admissibility of the reflected dividend policy and the integrability and transversality properties needed on the infinite horizon. 
    Together, these results prove the existence of a regime-dependent equilibrium barrier policy in the two-state bi-exponential model.
    
    \item We investigate how shareholder patience and the economic environment affect the equilibrium barriers numerically. 
    The experiments distinguish changes in the average discount rate from changes in discount-rate dispersion at a fixed initial average rate. 
    We also examine transition speeds, regime composition, profitability, and volatility. 
    In the parameter ranges considered, switching speed can have a nonmonotone effect on the barriers and can change their ordering, illustrating how expectations about future regimes interact with current dividend incentives.

\end{enumerate}
\vskip 5pt
The remainder of the paper is organized as follows. 
Section~\ref{sec:model-formulation} formulates the model and the equilibrium concept. 
Section~\ref{sec:verification} presents the general verification theorem. 
Section~\ref{sec:special-case} proves the existence of the equilibrium for the two-state bi-exponential model. 
Section~\ref{sec:numerical-illustrations} presents the numerical results.

\section{Model formulation}
\label{sec:model-formulation}

Let $(\Omega,\mathcal F,\mathbb F,\mathbb P)$ be a complete filtered probability space supporting a standard Brownian motion $B = \{B_t\}_{t\ge0}$ and an independent continuous-time Markov chain $\varepsilon=\{\varepsilon_t\}_{t\ge0}$ with state space $S=\{1,\cdots,N\}$ and generator $Q=(q_{ij})_{i,j\in S}$, where $q_{ij}\geq0\,(i\neq j)$, $q_{ii}=-\sum_{j\neq i}q_{ij}$.
The filtration $\mathbb F=\{\mathcal F_t\}_{t\geq0}$ is the usual augmentation of the natural filtration of $(B,\varepsilon)$ and represents the information available to the firm.

\subsection{Surplus dynamics and objective function}
\label{subsec:surplus-objective}

We consider a firm whose operating cash flows depend on the observable regime $\varepsilon$. 
The regime describes macroeconomic or industry conditions, while $B$ captures cash-flow fluctuations within each regime.
Dividend decisions do not affect the transition intensities. This setting follows the regime-switching dividend models of \citet{sotomayor2011classical} and \citet{jiang2012optimal}.

A singular control $\xi = \{\xi_t\}_{t\ge0}$ is an $\mathbb F$-adapted, nondecreasing, c\`adl\`ag process with $\xi_{0-}=0$. Here $\xi_t$ is the cumulative dividend paid by time $t$, and an initial payment $\Delta\xi_0=\xi_0$ is allowed.
For initial surplus $x\geq0$, the firm's controlled surplus $X = \{X_t\}_{t\ge0}$ satisfies
\begin{equation}
\label{eq:surplus}
X_t^\xi
=x+\int_0^{t\wedge\tau^\xi}
       \mu_{\varepsilon_{r}}(X_{r}^\xi)\,dr
  +\int_0^{t\wedge\tau^\xi}
       \sigma_{\varepsilon_{r}}(X_{r}^\xi)\,dB_r
  -\xi_{t\wedge\tau^\xi},
\qquad X_{0-}^\xi=x,\, \varepsilon_0 = i,
\end{equation}
where $\tau^\xi:=\inf\{t\geq0:X_t^\xi=0\}$ is the bankruptcy time, and $\inf\varnothing=+\infty$.
In regime $i$, $\mu_i(x)$ is the expected rate of net cash inflow and $\sigma_i(x)$ measures cash-flow volatility at surplus $x$. 

For convenience, we make the following assumption: 
\begin{assumption}
\label{ass:coefficient}
for each $i\in S$, $\mu_i$ and $\sigma_i$ are Lipschitz continuous on $[0,+\infty)$ and $\sigma_i(x)>0$ for $x>0$.
\end{assumption}
Let $\beta:[0,+\infty)\to(0,+\infty)$ denote the discount function of shareholders, where $\beta\in C^1([0,\infty)),\,\beta(0)=1$ and $\beta'(t)\leq0$.
This paper considers the following discounted cumulative dividend as the objective function:
\begin{equation}
\label{eq:dividend-payoff}
J(x,i;\xi)
:=\mathbb E_{x,i}\!\left[
\int_{[0,\tau^\xi]}\beta(t)\,d\xi_t
\right],
\end{equation}
where $\mathbb E_{x,i}(\cdot) = \mathbb E[\cdot| X_{0-} = x,\varepsilon_0 = i]$ and $\int_{[0,\tau^\xi]} \beta(t)\,d\xi_t:= \int_0^{\tau^\xi}\beta(t)\,d\xi^c_t + \sum_{t\in[0,\tau^\xi]}\beta(t)\Delta \xi_t$.
\vskip 5pt
We consider maximizing the objective function, i.e., the following optimization problem: $sup\{J(x,i;\xi)\}$.
However, this problem is time-inconsistent because the discount function $\beta$ might not be exponential.
Therefore, this paper considers the equilibrium solution to this problem rather than maximizing it based on the concepts of the equilibrium framework in singular control proposed by \citet{liang2024equilibria}.

\begin{remark}
\label{rem:regime-brownian-independence}
Considering that $\{\varepsilon_t\}_{t\geq0}$ represents the firm's overall operating environment, directly influencing the expected rate of net cash inflow and cash-flow volatility, and $\{B_t\}_{t\geq0}$ captures cash-flow fluctuations within a given regime, it is reasonable to assume that $B$ and $\varepsilon$ are independent.
\end{remark}

\begin{remark}
\label{rem:simultaneous-jumps}
If a regime switch and a dividend jump occur at the same time, we assume that the regime switches first and the dividend is chosen after the new regime is observed.
Thus, if $X_{t-}^{\xi}=x$, $\varepsilon_{t-}=i$, $\varepsilon_t=j$, and
$\Delta\xi_t=a$, the state changes as
$$
(x,i)\longrightarrow(x,j)\longrightarrow(x-a,j),
\qquad 0\leq a\leq x.
$$
The regime switch itself does not change the surplus.
This assumption is reasonable, given that the regime is observable to the firm.
\end{remark}

\begin{remark}
\label{rem:shareholder-discounting}
\citet{ebert2020weighted} motivate non-exponential discounting through the aggregation of heterogeneous time preferences.
For example, shareholders with a long investment horizon may value delayed dividends more highly than shareholders with immediate liquidity needs.
Assigning fixed evaluation weights $\theta$ and $1-\theta$ to two shareholder groups with respective discount rates $0<\delta_1<\delta_2$ gives $\beta(t)=\theta e^{-\delta_1t}+(1-\theta)e^{-\delta_2t},\; 0<\theta<1$, which is a standard non-exponential discount function.
\end{remark}

\subsection{Admissible controls and control laws}
\label{subsec:admissible-laws}
Motivated by the equilibrium framework of singular control proposed by \citet{liang2024equilibria}, we introduce the following definitions.
\begin{definition}[Admissible singular control]
\label{def:admissible-control}
Given $(x,i)\in E:=[0,+\infty)\times S$, a singular control $\xi$ is an admissible singular control if the following (a) and (b) hold:

\medskip
\noindent
(a) the surplus dynamics
$$
dX_t^\xi=\mu_{\varepsilon_{t}}(X_{t}^\xi)\,dt
+\sigma_{\varepsilon_{t}}(X_{t}^\xi)\,dB_t
-d\xi_{t},\;t\in [0,\tau^\xi],\,X^\xi_{0-} = x,\;\;\varepsilon_0 = i.
$$
has a unique strong solution $\{X_t^\xi\}_{t\in [0,\tau^\xi]}$ and $X_t^\xi\geq0$, $0\leq\Delta\xi_t\leq X_{t-}^\xi$. 

\medskip
\noindent
(b) $J(x,i;\xi)=\mathbb E_{x,i}[\int_{[0,\tau^\xi]}\beta(t)\,d\xi_t]<+\infty$.

We denote the set of all admissible singular controls by $\mathcal D(x,i)$.
\end{definition}

An admissible singular control is defined for a given initial state. 
To specify the firm's response after a perturbation (see next subsection~\ref{subsec:equilibrium}), we define a singular control law as a partition of the state-regime space $E$, following \citet{liang2024equilibria}. 

\begin{definition}[Admissible singular control law]
\label{def:admissible-law}
Suppose that $\Xi=(W^\Xi,P^\Xi)$ is a partition of $E$, where $W^\Xi$ is a relative open set of $E$ denoting the waiting region and $P^\Xi=E\setminus W^\Xi$ is a relative closed set of $E$ denoting the payout region. 
We call $\Xi$ an admissible singular control law if the following (a) and (b) hold.

\medskip
\noindent
(a) For any initial $(x,i)\in E$, the Skorokhod reflection problem
\begin{equation}
\label{eq:state-reflection}
\left\{
\begin{aligned}
&dX_t^\xi=\mu_{\varepsilon_{t}}(X_{t}^\xi)\,dt
+\sigma_{\varepsilon_{t}}(X_{t}^\xi)\,dB_t
-d\xi_{t},\;t\in [0,\tau^\xi],\\
&(X_t^\xi,\varepsilon_t)
\in\overline{W^\Xi},
\;t\in [0,\tau^\xi],\\
&\int_0^{t\wedge\tau^\xi}
\mathbf1_{\{(X_r^\xi,\varepsilon_r)\notin P^\Xi\}}
\,d(\xi_r)^c=0,
\;t\geq0,\\
&\Delta\xi_t
=\inf\left\{
a\in[0,X_{t-}^\xi]:
(X_{t-}^\xi-a,\varepsilon_t)
\in\overline{W^\Xi})
\right\},\\
&X^\xi_{0-} = x,\;\;\varepsilon_0 = i 
\end{aligned}
\right.
\end{equation}
has a unique strong solution $(X^\xi,\xi):=(X^{(x,i,\Xi)},\xi^{(x,i,\Xi)})$. We call $\xi$ the singular control generated by $\Xi$ at $(x,i)$.

\medskip
\noindent
(b) $\xi$ generated by $\Xi$ at $(x,i)$ is an admissible singular control, i.e., $\xi^{(x,i,\Xi)}\in \mathcal{D}(x,i)\in E$.
\end{definition}

The definition of admissible singular control law corresponds to the whole state-regime space $E$. It is different from admissible control $\mathcal{D}(x,i)$ (Definition~\ref{def:admissible-control}), which is defined for each initial $(t,x)\in E$.
However, there are some relations. The singular control generated by an admissible singular law at $(x,i)$ is indeed an admissible singular control in $\mathcal{D}(x,i)$.
Hence, an admissible singular control law corresponds to a certain combination of admissible singular controls by varying the initial $(x,i)$.

\subsection{Equilibrium singular control laws}
\label{subsec:equilibrium}
In this subsection, based on the admissible definition~\ref{def:admissible-control} and ~\ref{def:admissible-law}, we introduce the equilibrium singular control law proposed by \citet{liang2024equilibria}.

Let $\widehat\Xi$ be an admissible law. Fix $(x,i)\in E$,
$\eta\in\mathcal D(x,i)$, and $h>0$. The perturbed control follows
$\eta$ on $[0,h)$ and then restarts $\widehat\Xi$ from
$(X_{h-}^\eta,\varepsilon_h)$. Denote the incremental control generated
at this restart by $\zeta^{h,X_{h-}^\eta,\varepsilon_h,\widehat\Xi}$,
normalized to be zero at $h-$. On $\{\tau^\eta\geq h\}$, set
\begin{equation}
\label{eq:perturbed-control}
\xi_t^h=
\begin{cases}
\eta_t, &0\leq t<h,\\[2pt]
\eta_{h-}+\zeta_t^{h,X_{h-}^\eta,\varepsilon_h,\widehat\Xi},
&t\geq h.
\end{cases}
\end{equation}
On $\{\tau^\eta<h\}$, the process remains stopped at $\tau^\eta$.
Any payment at $h$ is thus generated by $\widehat\Xi$ after the regime
at $h$ has been observed.

\begin{definition}[Equilibrium singular control law]
\label{def:equilibrium-law-revised}
An admissible law $\widehat\Xi$ is an equilibrium if, for every
$(x,i)\in E$ and $\eta\in\mathcal D(x,i)$, the controls $\xi^h$ in
\eqref{eq:perturbed-control} belong to $\mathcal D(x,i)$ for all
sufficiently small $h>0$ and
\begin{equation}
\label{eq:equilibrium-condition-revised}
\liminf_{h\downarrow0}
\frac{J(x,i;\xi^{x,i,\widehat\Xi})-J(x,i;\xi^h)}{h}
\geq0.
\end{equation}
The corresponding equilibrium value is
$V_i^{\widehat\Xi}(x):=J(x,i;\xi^{x,i,\widehat\Xi})$.
\end{definition}

This paper aims to determine the equilibrium singular control law of the objective function \eqref{eq:dividend-payoff}.

\begin{remark}
\label{rem:classical-dividends}
\citet{jiang2012optimal} study dividend distribution under regime switching with a regime-dependent discount rate.
Their discount factor has the form 
$$\exp(-\int_0^t r(\varepsilon_s)\,ds),\; r(i)>0,\; i\in S,$$
which is multiplicative along the regime path and yields a time-consistent control problem.
In the present model, the non-exponential function $\beta$ is applied, thus introducing time-inconsistency.
We therefore consider an equilibrium law.
\end{remark}

\begin{remark}
\label{rem:singular-framework}
\citet{liang2024equilibria} introduce equilibrium singular control laws for a finite-horizon problem with non-exponential discounting.
Our perturbation follows their construction: after a short deviation, the prescribed law is restarted from the state reached at the end of the deviation. 
The present problem has an infinite horizon, regime-dependent cash flows, and an absorbing bankruptcy boundary.
Moreover, the coefficients do not depend on cumulative payments,  thus $\xi$ need not be included as a separate state variable.
\end{remark}

\section{Verification theorem}
\label{sec:verification}

In this section, we establish a verification theorem for the equilibrium singular control law.
The assumption~\ref{ass:coefficient} remains in force.
For convenience, we extend the surplus process $X$ by absorption condition: $X_t^\xi=0$, $\xi_t=\xi_{\tau^\xi}$ for $t\geq\tau^\xi$.

For $\varphi=(\varphi_i)_{i\in S}$ with $\varphi_i\in C^2((0,+\infty))$, the infinitesimal generator of the uncontrolled surplus and regime process is
\begin{equation}
\label{eq:verification-generator}
(\mathcal L\varphi)_i(x)
=\mu_i(x)\varphi_i'(x)
+\frac12\sigma_i^2(x)\varphi_i''(x)
+\sum_{j\neq i}q_{ij}\bigl(\varphi_j(x)-\varphi_i(x)\bigr),
\qquad x>0.
\end{equation}
Let $V=(V_i:x\in[0,+\infty)\to\mathbb{R})_{i\in S}$ and $f=(f_i:(x,s)\in[0,+\infty)\times[0,+\infty)\to\mathbb{R})_{i\in S}$ be candidate functions.
The variable $s$ denotes a discount lag. Define
$$
\begin{aligned}
W&=(\{0\}\times S)
\cup\{(x,i)\in E:x>0,\ V_i'(x)>1\},\quad
P=E\setminus W,\\
W_i&=\{x>0:(x,i)\in W\},\qquad
P_i=\{x>0:(x,i)\in P\},\qquad
\widehat\Xi=(W,P).
\end{aligned}
$$
The absorbing boundary is included in $W$. 
All differential equations below are imposed only for $x>0$, and closures are taken relative to $E$.

\begin{theorem}[Verification theorem]
\label{thm:verification}
Suppose that $\widehat\Xi$ is an admissible singular control law and that $V$ and $f$ satisfy the following conditions:

\medskip
\noindent
(a) For every $i\in S$,
$$
\begin{aligned}
&V_i\in C^2((0,+\infty))\cap C([0,+\infty)),\\
&f_i\in C^{0,1}([0,+\infty)^2)
\cap C^{1,1}((0,+\infty)\times[0,+\infty))
\cap C^{2,1}((\overline{W_i}\cap(0,+\infty))\times[0,+\infty)),
\end{aligned}
$$
and $V_i(x)=f_i(x,0),\;f_i(0,s)=0,\; x,s\geq0$.
Derivatives on $\overline{W_i}$ are taken from $W_i$, and those at $s=0$ are right derivatives.

\medskip
\noindent
(b) For every $i\in S$, the extended HJB system holds:
\begin{equation}
\label{eq:verification-hjb}
\left\{
\begin{aligned}
&\max\bigl\{(\mathcal LV)_i(x)+f_{i,s}(x,0),\,1-V_i'(x)\bigr\}=0,
&&x>0,\\
&f_{i,s}(x,s)+(\mathcal Lf(\cdot,s))_i(x)=0,
&&(x,i)\in\overline W,\ x>0,\ s\geq0,\\
&f_{i,x}(x,s)=\beta(s),
&&(x,i)\in P,\ s\geq0.
\end{aligned}
\right.
\end{equation}
At a payout interface, $f_{i, xx}$ in the second equation is evaluated by its continuous extension from $W_i$.

\medskip
\noindent
(c) There exists $p\geq1$ such that, for every $T>0$, some constant $C_T>0$ satisfies
\begin{equation}
\label{eq:verification-growth}
\max_{i\in S}\left\{
|V_i'(x)|+|(\mathcal LV)_i(x)|+\sup_{0\leq s\leq T}|f_{i,s}(x,s)|
\right\}
\leq C_T(1+x^p),\qquad x>0.
\end{equation}
For every $(x,i)\in E$ and $s\geq0$, the control $\widehat\xi=\xi^{x,i,\widehat\Xi}$, its surplus $\widehat X=X^{x,i,\widehat\Xi}$ and its bankruptcy time $\widehat\tau$ satisfy
\begin{equation}
\label{eq:verification-transversality}
\lim_{T\to\infty}
\mathbb E_{x,i}\!\left[
\left|f_{\varepsilon_{T\wedge\widehat\tau}}
(\widehat X_{T\wedge\widehat\tau},s+T\wedge\widehat\tau)\right|
\right]=0.
\end{equation}

Then $\widehat\Xi$ is an equilibrium singular control law. In addition,
\begin{equation}
\label{eq:verification-representation}
f_i(x,s)=\mathbb E_{x,i}\!\left[
\int_{[0,\widehat\tau]}\beta(s+t)\,d\widehat\xi_t
\right],\qquad
V_i(x)=f_i(x,0)=J(x,i;\widehat\xi).
\end{equation}
\end{theorem}

\begin{proof}
The assertions are immediate when $x=0$. Fix $x>0$ and $i\in S$.
The Lipschitz continuity on the coefficients in Assumption~\ref{ass:coefficient} implies linear growth.
Because  dividends $\xi$ are nondecreasing and the corresponding surplus $X^\xi$ is nonnegative, the Burkholder-Davis-Gundy and Gronwall inequalities give, for every $\xi\in\mathcal D(x,i)$, $T>0$, and $m\geq2$,
$$
\mathbb E_{x,i}\!\left[
\sup_{0\leq t\leq T}|X_t^\xi|^m+\xi_T^m
\right]\leq C_{m,T}(1+x^m),
$$
where $C_{m,T}$ does not depend on $\xi$. 
Also, \eqref{eq:verification-growth}, $V_i(0)=0$, and
$f_i(x,s)=V_i(x)+\int_0^s f_{i,s}(x,r)\,dr$ give
$$
\max_{i\in S}\left\{|V_i(x)|+
\sup_{0\leq s\leq T}|f_i(x,s)|\right\}
\leq C_T(1+x^{p+1}),\qquad x\geq0.
$$
These estimates will justify the finite-horizon limits below.

First, we establish the representation of $f$. 
Fix $s\geq0$ and write $\rho_T=T\wedge\widehat\tau$.
Because $f_{i}$ is $C^1$ continuous in $x\in(0,+\infty)$, the following generalized It\^o formula applies across payout interfaces without a local-time term. 
Applying generalized It\^o formula to $f_{\varepsilon_t}(\widehat X_t,s+t)$ on $[0-,\rho_T]$ gives, after localization and taking expectations,
$$
\begin{aligned}
&\mathbb E_{x,i}\!\left[
  f_{\varepsilon_{\rho_T}}(\widehat X_{\rho_T},s+\rho_T)
  \right]-f_i(x,s)
=\mathbb E_{x,i}\int_0^{\rho_T}
  \left[f_{\varepsilon_t,s}(\widehat X_t,s+t)
  +(\mathcal Lf(\cdot,s+t))_{\varepsilon_t}(\widehat X_t)\right]dt\\
&\qquad-\mathbb E_{x,i}\int_0^{\rho_T}
  f_{\varepsilon_t,x}(\widehat X_t,s+t)\,d\widehat\xi_t^c
-\mathbb E_{x,i}
  \sum_{t\in[0,\rho_T]}\int_0^{\Delta\widehat\xi_t}
  f_{\varepsilon_t,x}(\widehat X_{t-}-a,s+t)\,da.
\end{aligned}
$$
Here the regime jump is evaluated before a simultaneous dividend jump; the latter therefore uses the new regime $\varepsilon_t$.
The drift integral vanishes by the state constraint and
\eqref{eq:verification-hjb}.
The continuous dividend term equals $-\mathbb E_{x,i}\int_0^{\rho_T}\beta(s+t)\,d\widehat\xi_t^c$.
For any $\Delta\widehat\xi_t>0$, minimal intervention in \eqref{eq:state-reflection} gives $(\widehat X_{t-}-a,\varepsilon_t)\in P,\; 0\leq a<\Delta\widehat\xi_t$.
Hence the jump term equals $-\mathbb E_{x,i}\sum_{t\in[0,\rho_T]}
\beta(s+t)\Delta\widehat\xi_t$.
The localized terminal values and dividend integrals are dominated by an integrable multiple of $1+\sup_{t\leq T}(\widehat X_t)^{p+1}+\widehat\xi_T$.
We may therefore remove the localization and obtain
\begin{equation}
\label{eq:verification-finite-horizon}
f_i(x,s)=\mathbb E_{x,i}\int_{[0,\rho_T]}
\beta(s+t)\,d\widehat\xi_t
+\mathbb E_{x,i}\!\left[
f_{\varepsilon_{\rho_T}}(\widehat X_{\rho_T},s+\rho_T)
\right].
\end{equation}
Letting $T\to\infty$ in \eqref{eq:verification-finite-horizon}, monotone convergence and \eqref{eq:verification-transversality} establish the first identity in \eqref{eq:verification-representation}. 
Setting $s=0$ establishes the second.

Next, fix $\eta\in\mathcal D(x,i)$ and let $\xi^h$ be the perturbed control in \eqref{eq:perturbed-control}.
Put $a=\eta_0$ and $z=x-a$.
Under $\mathbb P_{x,i}$, the augmented natural filtration $\mathbb{F}$ is trivial at time zero, as such,  $a$ is deterministic up to a null set.
If $z=0$, the deviation pays $x$ immediately and stops. 
As \eqref{eq:verification-hjb} implies $V_i'\geq1$ and $V_i(0)=0$,
$$
J(x,i;\xi^h)-V_i(x)=x-V_i(x)\leq0.
$$
It remains to consider $z>0$.

By restarting $\widehat\Xi$ at time $h$ and using \eqref{eq:verification-representation},
\begin{equation}
\label{eq:verification-perturbed-payoff}
J(x,i;\xi^h)
=\mathbb E_{x,i}\!\left[
\int_{[0,h)}\beta(t)\,d\eta_t
+f_{\varepsilon_h}(X_{h-}^\eta,h)
\right].
\end{equation}
If bankruptcy occurs before $h$, the second term is zero by absorption and $f_j(0,h)=0$. 
The moment and growth estimates above show that this payoff is finite. 
The remaining admissibility conditions follow from concatenation and the admissibility of $\widehat\Xi$.
Thus $\xi^h\in\mathcal D(x,i)$.

Using right continuity of $\eta$, we have
$$
\tau^\eta>0,\qquad X_{h-}^\eta\longrightarrow z,\qquad
\varepsilon_h\longrightarrow i,
\qquad h\downarrow0,\quad\mathbb P_{x,i}\text{-a.s.}
$$
Because of $f_j(y,h)-V_j(y)=\int_0^h f_{j,s}(y,r)\,dr$, joint continuity, \eqref{eq:verification-growth} and dominated convergence give
$$
\mathbb E_{x,i}\!\left[
f_{\varepsilon_h}(X_{h-}^\eta,h)
-V_{\varepsilon_h}(X_{h-}^\eta)
\right]=h f_{i,s}(z,0)+o(h).
$$
Similarly, continuity of $(\mathcal LV)_i$ on $(0,\infty)$ gives
$$
\mathbb E_{x,i}\int_0^h\mathbf1_{\{t<\tau^\eta\}}
(\mathcal LV)_{\varepsilon_t}(X_t^\eta)\,dt
=h(\mathcal LV)_i(z)+o(h).
$$
For both limits, the normalized averages are dominated by an integrable multiple of $1+\sup_{t\leq1}(X_t^\eta)^p$.

Applying the generalized It\^o formula to $V_{\varepsilon_t}(X_t^\eta)$ on $[0-,h)$, we have
$$
\begin{aligned}
&\mathbb E_{x,i}\!\left[
V_{\varepsilon_h}(X_{h-}^\eta)-V_i(x)
+\int_{[0,h)}\beta(t)\,d\eta_t
\right]\quad=\mathbb E_{x,i}\int_0^h\mathbf1_{\{t<\tau^\eta\}}
  (\mathcal LV)_{\varepsilon_t}(X_t^\eta)\,dt\\
&+\mathbb E_{x,i}\int_0^{h\wedge\tau^\eta}
  \bigl[\beta(t)-V_{\varepsilon_t}'(X_t^\eta)\bigr]\,d\eta_t^c+\mathbb E_{x,i}\sum_{t\in[0,h)}
  \int_0^{\Delta\eta_t}
  \bigl[\beta(t)-V_{\varepsilon_t}'(X_{t-}^\eta-a)\bigr]\,da.
\end{aligned}
$$
Similarly, localization can again be removed: the drift term is polynomially bounded, and the control terms are bounded in absolute value by an integrable multiple of $(1+\sup_{t\leq h}(X_t^\eta)^p)\eta_h$.
As $\beta(t)\leq1$ and $V_j'\geq1$, both control terms are
nonpositive. 
Combining this inequality with \eqref{eq:verification-perturbed-payoff} and the preceding limits gives
$$
J(x,i;\xi^h)-V_i(x)
\leq h\bigl[(\mathcal LV)_i(z)+f_{i,s}(z,0)\bigr]+o(h)
\leq o(h),
$$
where the last inequality follows from \eqref{eq:verification-hjb}. 
By \eqref{eq:verification-representation}, we have 
$$
\liminf_{h\downarrow0}
\frac{J(x,i;\xi^{x,i,\widehat\Xi})-J(x,i;\xi^h)}{h}
\geq0.
$$
Therefore $\widehat\Xi$ is an equilibrium singular control law in the sense of Definition~\ref{def:equilibrium-law-revised}.
\end{proof}

\section{A special case: two regimes and bi-exponential discount}
\label{sec:special-case}

In this section, we consider two regimes with constant cash-flow coefficients and bi-exponential discounting. 
We seek a barrier law and reduce the equalities in \eqref{eq:verification-hjb} to a free-boundary problem.

\subsection{The Special model}
\label{subsec:special-model}

Let $S=\{1,2\}$ and
$$
Q=\begin{pmatrix}-q_1&q_1\\q_2&-q_2\end{pmatrix},
\qquad q_1,q_2>0.
$$
For $i\in S$, assume $\mu_i(x)\equiv\mu_i>0$ and $\sigma_i(x)\equiv\sigma_i>0$, and set $a_i=\sigma_i^2/2$ and $j=3-i$. 
The firm's cash-flow characteristics thus depend on the operating regime but not on its cash reserves. 
The expected duration of regime $i$ is $1/q_i$. 
Before bankruptcy, the surplus satisfies
$$
dX_t^\xi=\mu_{\varepsilon_t}\,dt
+\sigma_{\varepsilon_t}\,dB_t-d\xi_t.
$$
The discount function is
\begin{equation}
\label{eq:special-discount}
\beta(s)=\theta e^{-\delta_1s}+(1-\theta)e^{-\delta_2s}
=\sum_{\ell=1}^2\omega_\ell e^{-\delta_\ell s},
\qquad s\geq0,
\end{equation}
where $0<\theta<1$, $0<\delta_1<\delta_2$, $\omega_1=\theta$, and $\omega_2=1-\theta$.
Here $\theta$ is the weight assigned to the more patient shareholder group, whose discount rate is $\delta_1$, following the weighted discounting interpretation of \citet{ebert2020weighted}.

For $b=(b_1,b_2)\in(0,\infty)^2$, consider the barrier law $\Xi^b=(W^b,P^b)$, where
\begin{equation}
\label{eq:special-barrier-law}
W^b=\bigcup_{i=1}^2\bigl([0,b_i)\times\{i\}\bigr),
\qquad P^b=E\setminus W^b.
\end{equation}
In regime $i$, the firm pays any surplus above $b_i$ immediately and reflects the surplus downward at $b_i$. 
A switch to a lower barrier therefore triggers a lump-sum dividend if the surplus exceeds that barrier.

For this partition, the auxiliary equations in \eqref{eq:verification-hjb} become, for $i\in S$ and $s\geq0$,
\begin{equation}
\label{eq:special-auxiliary-system}
\left\{
\begin{aligned}
&f_{i,s}+a_if_{i,xx}+\mu_if_{i,x}+q_i(f_j-f_i)=0,
&&0<x\leq b_i,\\
&f_{i,x}(x,s)=\beta(s),&&x\geq b_i,\\
&f_i(0,s)=0.
\end{aligned}
\right.
\end{equation}
At $x=b_i$, second derivatives are taken from the waiting region.

\subsection{Ansatz}
\label{subsec:special-ansatz}

For a barrier vector $b=(b_1,b_2)$, we make the ansatz
\begin{equation}
\label{eq:special-ansatz}
f_i^b(x,s)=\sum_{\ell=1}^2\omega_\ell e^{-\delta_\ell s}
u_i^{(\ell),b}(x),\qquad i\in S.
\end{equation}
Then
$$
V_i^b(x)=\sum_{\ell=1}^2\omega_\ell u_i^{(\ell),b}(x),
\qquad
f_{i,s}^b(x,0)=-\sum_{\ell=1}^2\omega_\ell\delta_\ell
u_i^{(\ell),b}(x).
$$
Substituting \eqref{eq:special-ansatz} into \eqref{eq:special-auxiliary-system} and using the linear independence of $e^{-\delta_1s}$ and $e^{-\delta_2s}$ gives the fixed-boundary problem
\begin{equation}
\label{eq:special-fixed-boundary}
\left\{
\begin{aligned}
&a_i\bigl(u_i^{(\ell),b}\bigr)''(x)
+\mu_i\bigl(u_i^{(\ell),b}\bigr)'(x)
+q_i\bigl(u_j^{(\ell),b}(x)-u_i^{(\ell),b}(x)\bigr)
-\delta_\ell u_i^{(\ell),b}(x)=0,
&&0<x\leq b_i,\\
&u_i^{(\ell),b}(0)=0,\qquad
\bigl(u_i^{(\ell),b}\bigr)'(b_i-)=1,\\
&u_i^{(\ell),b}(x)=u_i^{(\ell),b}(b_i)+x-b_i,
&&x\geq b_i,
\end{aligned}
\right.
\end{equation}
for $i,\ell\in\{1,2\}$ and $j=3-i$. 
We seek solutions in $C^1([0,\infty))\cap C^2((0,b_i))$, with second derivatives at $b_i$ understood from the left. 

The affine extension in \eqref{eq:special-fixed-boundary} is used in the coupling term whenever $x\geq b_j$. 
Thus, if $b_1<b_2$, the equations are coupled on $(0,b_1)$, while on $(b_1,b_2)$ the first component $u_1^{(\ell),b}$ is affine and the second component $u_2^{(\ell),b}$ solves a scalar ODE.
The second component $u_2^{(\ell),b}$ and its first derivative match at $b_1$; its second derivative at $b_1$ is continuous by the ODE. 
Interchanging the indices covers $b_2<b_1$, and the intermediate interval is empty when $b_1=b_2$.

It remains to determine the barriers $b$. 
As $V_i^b$ is affine for $x\geq b_i$, the $C^2$ regularity required in Theorem~\ref{thm:verification} gives the equations
\begin{equation}
\label{eq:special-free-boundary}
\Gamma_i(b):=(V_i^b)''(b_i-)
=\sum_{\ell=1}^2\omega_\ell
\bigl(u_i^{(\ell),b}\bigr)''(b_i-)=0,
\qquad i=1,2.
\end{equation}
These conditions are imposed on the weighted values $V_i^b$.
For each fixed $b$, \eqref{eq:special-fixed-boundary} is a linear ODE.
The free-boundary problem is therefore to find $b^*\in(0,\infty)^2$ such that $\Gamma(b^*)=0$, where $\Gamma=(\Gamma_1,\Gamma_2)$.
The solvability of the fixed-boundary problem and the existence of such a zero will be established below.

Additionally, for a solution $b^*$, set $V_i=V_i^{b^*}$, $f_i=f_i^{b^*}$, and $\widehat\Xi=\Xi^{b^*}$. 
The remaining inequalities are
\begin{equation}
\label{eq:special-remaining-inequalities}
V_i'(x)\geq1,\quad 0<x<b_i^*,
\qquad
(\mathcal LV)_i(x)+f_{i,s}(x,0)\leq0,\quad x>b_i^*,
\qquad i=1,2.
\end{equation}
These inequalities and the other conditions of Theorem~\ref{thm:verification} will be verified after the existence analysis.

\subsection{Existence of a positive free-boundary pair}
\label{subsec:existence-free-boundaries}

We prove that the fixed-boundary problem \eqref{eq:special-fixed-boundary} has a unique solution and that $\Gamma(b)=0$ has a solution in $(0,\infty)^2$. 
Our main result is the following.

\begin{theorem}[Existence of the free boundaries]
\label{thm:existence-free-boundaries}
There exists $b^*\in(0,+\infty)^2$ such that $\Gamma(b^*)=0$.
Hence \eqref{eq:special-fixed-boundary} and \eqref{eq:special-free-boundary} admit a solution.
\end{theorem}

To prove Theorem~\ref{thm:existence-free-boundaries}, we first establish the following lemmas. 
Throughout, endpoint derivatives are taken from the interval on which the corresponding ODE holds.

\begin{lemma}[Solvability of the fixed-boundary problem]
\label{lem:existence-fixed-boundary}
For every $b\in(0,+\infty)^2$ and $\ell\in\{1,2\}$, \eqref{eq:special-fixed-boundary} has a unique solution $u^{(\ell),b}=(u_1^{(\ell),b},u_2^{(\ell),b})$.
Each component belongs to $C^1([0,+\infty))\cap C^3((0,b_i))$ and is smooth on the intervals determined by $b_1$ and $b_2$.
\end{lemma}

\begin{proof}
Fix $\ell$ and suppress the superscripts. 
Start the coupled ODEs at $u_i(0)=0$, $u_i'(0)=\alpha_i$, for $i=1,2$.
When the smaller barrier is reached, continue that component with its first derivative held constant and solve the other equation up to the larger barrier. 
Linear ODE theory gives a unique solution for every $\alpha=(\alpha_1,\alpha_2)^\top$, and the derivatives at the two
barriers depend linearly on $\alpha$. 
Write
$$
M_\ell(b)\alpha=\bigl(u_1'(b_1-),u_2'(b_2-)\bigr)^\top.
$$
It suffices to show that the $2\times2$ matrix $M_\ell(b)$ is invertible.

Let $\alpha\in\ker M_\ell(b)$ and denote the associated solution by $h=(h_1,h_2)$. Then $h_i(0)=0$, $h_i'(b_i-)=0$, and each $h_i$ is constant for $x\geq b_i$.
Suppose that $m=\max_{k\in S,\,x\geq0}h_k(x)>0$.
The maximum is attained by some $h_k$ at a point $x_0\in(0,b_k]$; if it is attained in the constant extension, take $x_0=b_k$.
At that point,
$$
h_k'(x_0)=0,\qquad h_k''(x_0)\leq0,\qquad h_j(x_0)\leq m.
$$
When $x_0=b_k$, the second inequality follows from a Taylor expansion from the left and $h_k'(b_k-)=0$.
The ODE, or its left limit at $b_k$, therefore gives
$$
0=a_kh_k''(x_0)+\mu_kh_k'(x_0)
+q_k\bigl(h_j(x_0)-h_k(x_0)\bigr)-\delta_\ell h_k(x_0)
\leq-\delta_\ell m<0,
$$
a contradiction. 
Applying the same argument to $-h$ gives $h=0$, and hence $\alpha=0$.
Thus $M_\ell(b)$ is invertible, and $\alpha=M_\ell(b)^{-1}(1,1)^\top$ gives the unique solution with unit slopes at the barriers. 
Its linear extensions are exactly those in \eqref{eq:special-fixed-boundary}.

The construction gives global $C^1$ regularity. 
On $(0,b_i)$, the ODE expresses $u_i''$ in terms of $u_i,u_i'$ and the globally $C^1$ function $u_j$, as such,  $u_i\in C^3((0,b_i))$.
The solution is smooth between the barriers by linear ODE theory.
We use the same maximum-minimum argument below without repeating the calculation at the extremum.
\end{proof}

\begin{lemma}[Continuity of $\Gamma$]
\label{lem:existence-continuity}
The map $\Gamma:(0,+\infty)^2\to\mathbb R^2$ defined in \eqref{eq:special-free-boundary} is continuous. 
It is real analytic on $\{b_1<b_2\}$ and on $\{b_2<b_1\}$.
\end{lemma}

\begin{proof}
For a fixed ordering of the barriers, the construction above uses constant-coefficient ODEs over intervals whose lengths depend linearly on $b$. 
Therefore, their solution matrices $M_\ell(b)$ are real analytic in $b$. 
Lemma~\ref{lem:existence-fixed-boundary} shows that the initial first derivative $\alpha=M_\ell(b)^{-1}(1,1)^\top$ is real analytic and thus $\Gamma$ is real analytic on $\{b_1<b_2\}$ and on $\{b_2<b_1\}$.

As $b\to(r,r)$, $r>0$, the solution matrix over the interval between the barriers tends to the identity. 
Both orderings therefore converge to the same problem with $b_1=b_2=r$. 
Its matrix $M_\ell(r,r)$ is invertible,  thus the initial first derivative $\alpha=M_\ell(b)^{-1}(1,1)^\top$ and boundary values $(u_i^{(\ell),b}(b_i), u_j^{(\ell),b}(b_i))$ are continuous also on the diagonal $\{b_1=b_2\}$. 
Finally, the ODE at $b_i$ gives
\begin{equation}
\label{eq:existence-boundary-identity}
a_i\bigl(u_i^{(\ell),b}\bigr)''(b_i-)
=-\mu_i+(q_i+\delta_\ell)u_i^{(\ell),b}(b_i)
-q_i u_j^{(\ell),b}(b_i).
\end{equation}
Therefore $\Gamma$ is continuous on the diagonal $\{b_1=b_2\}$ and furthermore $\Gamma$ is continuous on $(0,+\infty)^2$.
\end{proof}

\begin{lemma}[Uniform estimates]
\label{lem:existence-uniform}
There is a constant $C\geq1$, independent of $b\in(0,+\infty)^2$ and $\ell\in\{1,2\}$, such that
\begin{equation}
\label{eq:existence-uniform-bounds}
0\leq u_i^{(\ell),b}(x)\leq Cx,\qquad
0\leq\bigl(u_i^{(\ell),b}\bigr)'(x)\leq C,
\qquad x\geq0,\quad i=1,2.
\end{equation}
Moreover,
$$
\left|\bigl(u_i^{(\ell),b}\bigr)''(x)\right|
+\left|\bigl(u_i^{(\ell),b}\bigr)'''(x)\right|
\leq C,\qquad 0\leq x\leq b_i.
$$
\end{lemma}

\begin{proof}
Fix $b,\ell$ and write $u_i=u_i^{(\ell),b}$.
A negative minimum of $u$ cannot occur at zero or in a linear extension, and the maximum-minimum argument excludes it in the ODE region. Hence $u_i\geq0$.
Set $C_0=\max_k\mu_k/\delta_1$. Comparison with $x+C_0$ gives
$$
0\leq u_i(x)\leq x+C_0,\qquad x\geq0.
$$
Indeed, $u_i-x-C_0$ is negative at zero, has zero derivative at $b_i$, and its ODE has nonnegative right-hand side because $\delta_\ell(x+C_0)-\mu_i\geq0$.
A positive maximum is therefore excluded by the same argument.

Put $p_i=u_i'$ and extend $p_i=1$ for $x\geq b_i$.
Differentiating the ODE and taking its limit at zero gives
\begin{equation}
\label{eq:existence-derivative-system}
\begin{aligned}
&a_ip_i''+\mu_ip_i'+q_i(p_j-p_i)-\delta_\ell p_i=0,
&&0<x<b_i,\\
&p_i(b_i)=1,\qquad a_ip_i'(0)+\mu_ip_i(0)=0.
\end{aligned}
\end{equation}
Using $u_i\geq0$ and $u_i(0)=0$, we have $p_i(0)\geq0$.
The minimum argument applied to \eqref{eq:existence-derivative-system} now gives $p_i\geq0$.

We next bound $p_i(0)$ uniformly. 
Let $\gamma_i=\mu_i/a_i>0$.
The original ODE gives, for $0\leq x\leq\min\{1,b_i\}$,
$$
p_i(x)=e^{-\gamma_i x}p_i(0)
+\frac1{a_i}\int_0^x e^{-\gamma_i(x-y)}
\bigl[(q_i+\delta_\ell)u_i(y)-q_i u_j(y)\bigr]\,dy.
$$
Using  $0\leq u_k(y)\leq1+C_0$ for $0\leq y\leq1$, we have
$$
\left|p_i(x)-e^{-\gamma_i x}p_i(0)\right|
\leq \frac{(2q_i+\delta_2)(1+C_0)}{a_i}\,x
\leq Kx,\qquad 0\leq x\leq\min\{1,b_i\},
$$
where $K$ is independent of $b,\ell,i$.
If $b_i\leq1$, using $p_i(b_i)=1$ gives
$$
p_i(0)\leq e^{\gamma_i b_i}(1+Kb_i)
\leq e^{\gamma_i}(1+K).
$$
If $b_i>1$, integrating $e^{-\gamma_i x}p_i(0)\leq p_i(x)+Kx$ over $[0,1]$ gives
$$
\frac{1-e^{-\gamma_i}}{\gamma_i}\,p_i(0)
\leq \int_0^1p_i(x)\,dx+\frac K2
=u_i(1)+\frac K2
\leq1+C_0+\frac K2,
$$
and hence
$$
p_i(0)\leq
\frac{\gamma_i}{1-e^{-\gamma_i}}
\left(1+C_0+\frac K2\right).
$$
Taking the maximum of these bounds over $i=1,2$ gives $p_i(0)\leq C$, uniformly in $b,\ell,i$.
The maximum argument then gives $p_i\leq\max\{1,p_1(0),p_2(0)\}\leq C$.
Integrating $p_i$ from zero proves \eqref{eq:existence-uniform-bounds}.

Integrating \eqref{eq:existence-derivative-system} gives, for $0\leq x\leq b_i$,
$$
|p_i'(x)|\leq e^{-\gamma_i x}|p_i'(0)|
+\frac{C}{\gamma_i}(1-e^{-\gamma_i x})\leq C,
$$
where the last bound uses $p_i'(0)=-\gamma_i p_i(0)$.
The same ODE then bounds $p_i''$. 
This proves the remaining estimates.
\end{proof}

\begin{lemma}[Sign for small barriers]
\label{lem:existence-small}
Let $C$ be as in Lemma~\ref{lem:existence-uniform} and set
$$
r=\min_{i=1,2}\frac{\mu_i}{2C(q_i+\delta_2)}>0.
$$
For each $i\in S$, $0<b_i\leq r$ implies $\Gamma_i(b)<0$ for every $b_j>0$.
\end{lemma}

\begin{proof}
By \eqref{eq:existence-boundary-identity} and \eqref{eq:existence-uniform-bounds},
$$
a_i\bigl(u_i^{(\ell),b}\bigr)''(b_i-)
\leq-\mu_i+C(q_i+\delta_2)b_i\leq-\frac{\mu_i}{2}<0.
$$
The conclusion follows as $\omega_1,\omega_2\geq0$.
\end{proof}

\begin{lemma}[Sign for large barriers]
\label{lem:existence-large}
For each $i\in S$, there is $L_i>0$ such that
$$
b_i\geq L_i,\qquad 0<b_j\leq b_i
\quad\Longrightarrow\quad \Gamma_i(b)>0.
$$
\end{lemma}

\begin{proof}
It suffices to prove the assertion for each component $\bigl(u_i^{(\ell),b}\bigr)''(b_i-)$ and then take the larger of the two thresholds. 
Fix $i,\ell$. 
Suppose that there are barriers $b^{(n)}$ such that
$$
b_i^{(n)}\to\infty,\qquad
0<b_j^{(n)}\leq b_i^{(n)},\qquad
\bigl(u_i^{(\ell),b^{(n)}}\bigr)''(b_i^{(n)}-)\leq0.
$$
Write $s_n=b_j^{(n)}$, $d_n=b_i^{(n)}-b_j^{(n)}$, and $p_k^{(n)}=\bigl(u_k^{(\ell),b^{(n)}}\bigr)'$.
These derivatives satisfy \eqref{eq:existence-derivative-system}.
On $s_n<x<b_i^{(n)}$, $p_j^{(n)}=1$ and
\begin{equation}
\label{eq:existence-intermediate-derivative}
a_i(p_i^{(n)})''+\mu_i(p_i^{(n)})'
-(q_i+\delta_\ell)p_i^{(n)}+q_i=0.
\end{equation}
After taking a subsequence, either $d_n\to d<\infty$ or $d_n\to\infty$.

First suppose that $d_n\to d<\infty$. 
Then $s_n\to\infty$.
Choose $\kappa>0$ sufficiently small that $a_k\kappa^2+\mu_k\kappa<\delta_\ell$ for $k=i,j$, and set
$$
v_n(x)=C\bigl(e^{-\kappa x}+e^{-\kappa(s_n-x)}\bigr),
\qquad 0\leq x\leq s_n,
$$
where $C$ is the uniform bound in Lemma~\ref{lem:existence-uniform}. 
At the endpoints,
$$
v_n(0)=v_n(s_n)=C(1+e^{-\kappa s_n})
\geq\max\{p_k^{(n)}(0),p_k^{(n)}(s_n)\},
\qquad k=i,j.
$$
Moreover,
$$
\begin{aligned}
a_kv_n''+\mu_kv_n'-\delta_\ell v_n
=C(a_k\kappa^2-\mu_k\kappa-\delta_\ell)e^{-\kappa x}
+C(a_k\kappa^2+\mu_k\kappa-\delta_\ell)
e^{-\kappa(s_n-x)}
\leq0.
\end{aligned}
$$
The differences $p_k^{(n)}-v_n$ thus have nonpositive endpoint values and satisfy the corresponding coupled differential inequalities similar to \eqref{eq:existence-derivative-system} with nonnegative right-hand sides. 
The maximum-minimum argument excludes a positive maximum, giving
$$
0\leq p_k^{(n)}(x)
\leq C\bigl(e^{-\kappa x}+e^{-\kappa(s_n-x)}\bigr),
\qquad 0\leq x\leq s_n,\quad k=i,j.
$$

By Lemma~\ref{lem:existence-uniform}, there exists a constant $C$, independent of $n$, such that
$$
|p_k^{(n)}(x)|+|(p_k^{(n)})'(x)|+|(p_k^{(n)})''(x)|
\leq C,\qquad 0\leq x\leq s_n,\quad k=i,j,
$$
where derivatives at the endpoints are one-sided.
After the variable shift $z=x-s_n$, the common ODE interval becomes $[-s_n,0]$. 
As $s_n\to\infty$, every fixed interval $[-R,0]$ is contained in $[-s_n,0]$ for all sufficiently large $n$.
The bounds above imply that the shifted functions and their first derivatives are uniformly bounded and equicontinuous on $[-R,0]$.
The Arzel\`a-Ascoli theorem and a diagonal subsequence argument therefore give functions $\bar p_i,\bar p_j$ such that
$$
p_k^{(n)}(s_n+\cdot)\longrightarrow\bar p_k
\quad\text{in }C^1([-R,0])
\quad\text{for every }R>0,\qquad k=i,j.
$$
Writing the derivative equations in integral form and passing to the limit shows that $\bar p_i,\bar p_j$ are $C^2$ on $(-\infty,0)$ and satisfy
$$
a_k\bar p_k''+\mu_k\bar p_k'
+q_k(\bar p_{3-k}-\bar p_k)-\delta_\ell\bar p_k=0,
\qquad z<0,\quad k=i,j.
$$
Moreover, $p_j^{(n)}(s_n)=1$ gives $\bar p_j(0)=1$.

If $d>0$, extend $\bar p_j$ by $\bar p_j(z)=1$ for $0\leq z\leq d$.
Define $\bar p_i$ on $[0,d]$ as the solution of
$$
a_i\bar p_i''+\mu_i\bar p_i'
-(q_i+\delta_\ell)\bar p_i+q_i=0
$$
with initial value and derivative at zero equal to the left limits already constructed. 
The shifted functions $p_i^{(n)}(s_n+\cdot)$ satisfy the same scalar equation on $[0,d_n]$.
Continuous dependence of the scalar linear ODE on the initial data therefore gives
$$
p_i^{(n)}(s_n+\cdot)\longrightarrow\bar p_i
\quad\text{in }C^1([0,r]),
\qquad 0<r<d.
$$
To identify the value at $d$, use $p_i^{(n)}(s_n+d_n)=1$ and the uniform bound on $(p_i^{(n)})'$:
$$
\left|1-p_i^{(n)}(s_n+z)\right|
\leq C(d_n-z),\qquad 0\leq z\leq d_n.
$$
For each fixed $z<d$, letting $n\to\infty$ gives $|1-\bar p_i(z)|\leq C(d-z)$.
Then letting $z\uparrow d$ gives $\bar p_i(d)=1$.
If $d=0$, the same estimate at $z=0$ gives
$|1-p_i^{(n)}(s_n)|\leq Cd_n\to0$, and thus $\bar p_i(0)=1$.

Finally, for each fixed $z\leq0$, the preceding exponential bound gives, for all sufficiently large $n$,
$$
0\leq p_k^{(n)}(s_n+z)
\leq C\left(e^{-\kappa(s_n+z)}+e^{\kappa z}\right),
\qquad k=i,j.
$$
Then letting $n\to\infty$ gives
$$
0\leq\bar p_k(z)\leq Ce^{\kappa z},
\qquad z\leq0,\quad k=i,j,
$$
and hence $\lim_{z\to-\infty}\bar p_k(z)=0$.
If $d=0$, both components equal one at zero.
For $d>0$, continuity of $\bar p_i,\bar p_i'$ at zero and
$\bar p_j(0)=1$ give
$$
\bar p_i''(0-)
=\frac{-\mu_i\bar p_i'(0)+(q_i+\delta_\ell)\bar p_i(0)-q_i}{a_i}
=\bar p_i''(0+),
$$
and thus $\bar p_i$ is $C^2$ across zero.

The maximum argument applied to $\bar p_k-1$ gives $0\leq\bar p_i\leq1$ on $(-\infty,d]$ and $0\leq\bar p_j\leq1$ on $(-\infty,0]$.
Hence $\bar p_i'(d-)\geq0$.
If $\bar p_i'(d-)=0$, the ODE at $d$ give $a_i\bar p_i''(d-)=\delta_\ell$. Then a Taylor expansion gives
$$
\bar p_i(d-h)=1+\frac{\delta_\ell}{2a_i}h^2+o(h^2)>1
$$
for small $h>0$, a contradiction. 
Thus $\bar p_i'(d-)>0$.
Continuous dependence for \eqref{eq:existence-intermediate-derivative}, together with the
convergence at $z=0$, gives
$$
(p_i^{(n)})'(b_i^{(n)}-)\longrightarrow\bar p_i'(d-)>0.
$$
This contradicts the assumed nonpositive values.

Now suppose that $d_n\to+\infty$. 
Put
$$
c=\frac{q_i}{q_i+\delta_\ell},\qquad
\lambda_\pm=
\frac{-\mu_i\pm\sqrt{\mu_i^2+4a_i(q_i+\delta_\ell)}}{2a_i},
\qquad \lambda_-<0<\lambda_+.
$$
The scalar equation \eqref{eq:existence-intermediate-derivative} gives
$$
p_i^{(n)}(b_i^{(n)}+y)
=c+A_ne^{\lambda_+y}+B_ne^{\lambda_-y},
\qquad -d_n\leq y\leq0,
$$
where $A_n+B_n=1-c$ because $p_i^{(n)}(b_i^{(n)})=1$.
Evaluating at $y=-d_n$ gives
$$
B_n=
\frac{p_i^{(n)}(s_n)-c-(1-c)e^{-\lambda_+d_n}}
{e^{-\lambda_-d_n}-e^{-\lambda_+d_n}}
\longrightarrow0,
$$
as $p_i^{(n)}(s_n)$ is uniformly bounded.
Consequently,
$$
(p_i^{(n)})'(b_i^{(n)}-)
=(1-c)\lambda_++(\lambda_--\lambda_+)B_n
\longrightarrow(1-c)\lambda_+>0,
$$
which is again a contradiction.

Thus each component $\bigl(u_i^{(\ell),b}\bigr)''(b_i-)$ is positive for a sufficiently large $b_i$, uniformly over $0<b_j\leq b_i$. 
Therefore, their weighted sum $\Gamma_i$ is positive.
\end{proof}

\begin{proof}[Proof of Theorem~\ref{thm:existence-free-boundaries}]
Choose $B>\max\{r,L_1,L_2\}$, where $r$ and $L_i$ are supplied by Lemmas~\ref{lem:existence-small} and \ref{lem:existence-large}.
For every $t\in[r,B]$, these lemmas give
$$
\Gamma_1(r,t)<0<\Gamma_1(B,t),\qquad
\Gamma_2(t,r)<0<\Gamma_2(t,B).
$$
By Lemma~\ref{lem:existence-continuity}, $\Gamma$ is continuous on $[r,B]^2$. 
The Poincar\'e-Miranda theorem therefore gives $b^*\in[r,B]^2$ with $\Gamma(b^*)=0$.
\end{proof}

\subsection{Verification of the variational inequalities}
\label{subsec:verification-variational-inequalities}

Fix a free boundary pair $b^*$ from Theorem~\ref{thm:existence-free-boundaries}, and write $u_i^{(\ell)}=u_i^{(\ell),b^*}$.
Let $V_i=V_i^{b^*}$ and $f_i=f_i^{b^*}$ be given by \eqref{eq:special-ansatz}.
We now verify the remaining inequalities in \eqref{eq:special-remaining-inequalities} and identify the regions generated by $V$.

\begin{theorem}[Verification of the variational inequalities]
\label{thm:vi-verification}
For each $i\in S$, $V_i\in C^2((0,+\infty))$ and
\begin{equation}
\label{eq:vi-second-derivative-sign}
V_i''(x)\leq0,\qquad 0\leq x\leq b_i^*.
\end{equation}
Moreover,
\begin{equation}
\label{eq:vi-regions}
\{x>0:V_i'(x)>1\}=(0,b_i^*),\qquad
\{x>0:V_i'(x)=1\}=[b_i^*,\infty),
\end{equation}
and
\begin{equation}
\label{eq:vi-action-inequality}
(\mathcal LV)_i(x)+f_{i,s}(x,0)\leq0,
\qquad x\geq b_i^*.
\end{equation}
Consequently, $(V,f)$ satisfies \eqref{eq:verification-hjb}, and its waiting and payout regions are those of $\Xi^{b^*}$ in \eqref{eq:special-barrier-law}.
\end{theorem}

We first establish two lemmas. 
Set
$$
\Delta=\delta_2-\delta_1>0,\qquad
p_i^{(\ell)}=(u_i^{(\ell)})',\qquad
D_i=p_i^{(1)}-p_i^{(2)}.
$$
The affine extensions give $p_i^{(\ell)}=1$ and $D_i=0$ for $x\geq b_i^*$.

\begin{lemma}[Comparison of the different discount components]
\label{lem:vi-discount-comparison}
For each $i\in S$,
$$
u_i^{(1)}(x)\geq u_i^{(2)}(x)\geq0,\qquad x\geq0,
$$
and
\begin{equation}
\label{eq:vi-discount-order}
p_i^{(1)}(x)>p_i^{(2)}(x)>0,
\qquad 0<x<b_i^*.
\end{equation}
\end{lemma}

\begin{proof}
Lemma~\ref{lem:existence-uniform} gives $u_i^{(\ell)}\geq0$ and $p_i^{(\ell)}\geq0$.
We first exclude an interior zero of $p_i^{(\ell)}$.
If $p_i^{(\ell)}(x_0)=0$ for some $0<x_0<b_i^*$, then \eqref{eq:existence-derivative-system} gives
$$
(p_i^{(\ell)})'(x_0)=0,\qquad
0=a_i(p_i^{(\ell)})''(x_0)+q_ip_j^{(\ell)}(x_0).
$$
Both terms in the last equality are nonnegative, thus $p_j^{(\ell)}(x_0)=0$.
Using $p_j^{(\ell)}=1$ for $x\geq b_j^*$, we must have $x_0<b_j^*$, and then $(p_j^{(\ell)})'(x_0)=0$ as well.
Uniqueness for the coupled initial value problem with initial at $(p_i^{(\ell)}(x_0),(p_i^{(\ell)})'(x_0),p_j^{(\ell)}(x_0),(p_j^{(\ell)})'(x_0))=(0,0,0,0)$ forces both components $p_i^{(\ell)}$ to vanish on $(0,\min\{b_1^*,b_2^*\})$, which gives a contradiction. 
Thus $p_i^{(\ell)}>0$ on $(0,b_i^*)$.

Put $U_i=u_i^{(1)}-u_i^{(2)}$.
Subtracting the two equations in \eqref{eq:special-fixed-boundary} gives
$$
a_iU_i''+\mu_iU_i'+q_i(U_j-U_i)-\delta_1U_i
=-\Delta u_i^{(2)}\leq0,\qquad 0<x<b_i^*.
$$
Here $U_i(0)=0$, $U_i'(b_i^*{-})=0$, and $U_i$ is constant after $b_i^*$.
The similar minimum argument in Lemma~\ref{lem:existence-fixed-boundary} therefore gives $U_i\geq0$.

Using $D_i=U_i'$, we have $D_i(0)\geq0$ and $D_i(b_i^*)=0$.
Differentiating the equation for $U_i$ gives
\begin{equation}
\label{eq:vi-derivative-difference}
a_iD_i''+\mu_iD_i'+q_i(D_j-D_i)-\delta_1D_i
=-\Delta p_i^{(2)}<0,\qquad 0<x<b_i^*.
\end{equation}
The same minimum argument gives $D_i\geq0$, and we exclude an interior zero because the right-hand side is strictly negative.
This proves \eqref{eq:vi-discount-order}.
\end{proof}

\begin{lemma}[Signs of the second derivatives]
\label{lem:vi-second-derivatives}
For each $i\in S$,
$$
(u_i^{(1)})''(x)\leq0,\qquad V_i''(x)\leq0,
\qquad 0\leq x\leq b_i^*.
$$
\end{lemma}

\begin{proof}
On $[0,b_i^*]$, set $c_i=(u_i^{(1)})''$, with one-sided derivatives at the endpoints, and put $c_i=0$ for $x>b_i^*$.
By Lemma~\ref{lem:existence-fixed-boundary}, $c_i$ and $V_i''$ are
$C^1$ on $(0,b_i^*)$ and smooth away from $b_j^*$.
The free boundary condition and the affine extension also give $V_i''(b_i^*{-})=0=V_i''(b_i^*{+})$, thus  $V_i\in C^2((0,+\infty))$.

At zero, \eqref{eq:existence-derivative-system} gives
$$
c_i(0)=-\frac{\mu_i}{a_i}p_i^{(1)}(0)\leq0,
\qquad
V_i''(0)=-\frac{\mu_i}{a_i}V_i'(0)\leq0.
$$
Moreover, $D_i>0$ on $(0,b_i^*)$ and $D_i(b_i^*)=0$ imply $D_i'(b_i^*{-})\leq0$.
Using \eqref{eq:special-free-boundary}, we obtain
$$
c_i(b_i^*{-})=(1-\theta)D_i'(b_i^*{-})\leq0.
$$

Differentiating the fixed boundary equations twice gives
\begin{equation}
\label{eq:vi-second-derivative-system}
\begin{aligned}
&a_ic_i''+\mu_ic_i'+q_i(c_j-c_i)-\delta_1c_i=0,\\
&a_i(V_i'')''+\mu_i(V_i'')'
+q_i(V_j''-V_i'')-\delta_2V_i''
=-\theta\Delta c_i,
\end{aligned}
\end{equation}
for $0<x<b_i^*$, $x\neq b_j^*$.
The extensions $c_j=V_j''=0$ are used when $x>b_j^*$.
At an internal interface $b_j^*<b_i^*$, both $c_i$ and $V_i''$ are $C^1$, and we use the right limits of these equations.
If at a maximum there, their first derivatives vanish and their
right second derivatives are nonpositive, and therefore the same maximum argument applies.

As $c_i$ has nonpositive endpoint values, the first equation and the maximum argument give $c_i\leq0$.
The second equation then has a nonnegative right-hand side.
Together with $V_i''(0)\leq0$ and $V_i''(b_i^*)=0$, the same maximum argument gives $V_i''\leq0$.
\end{proof}

\begin{proof}[Proof of Theorem~\ref{thm:vi-verification}]
Lemma~\ref{lem:vi-second-derivatives} gives \eqref{eq:vi-second-derivative-sign} and the stated regularity.
Using $V_i'(b_i^*)=1$,
$$
V_i'(x)-1=-\int_x^{b_i^*}V_i''(y)\,dy\geq0,
\qquad 0<x<b_i^*.
$$
Suppose that $V_i'(x_0)=1$ for some $x_0\in(0,b_i^*)$.
Then $V_i'=1$ and $V_i''=0$ on $(x_0,b_i^*)$.
On each subinterval not containing $b_j^*$, the second equation in \eqref{eq:vi-second-derivative-system} gives
$$
0\geq q_iV_j''=-\theta\Delta c_i\geq0,
$$
and therefore $c_i=0$.
Continuity gives $c_i=0$ throughout $(x_0,b_i^*)$.
Using $V_i''=\theta c_i+(1-\theta)(u_i^{(2)})''=0$, both $(u_i^{(\ell)})''$ vanish there.
Then $p_i^{(\ell)}(b_i^*)=1$ gives $p_i^{(1)}=p_i^{(2)}=1$ on this interval, contradicting Lemma~\ref{lem:vi-discount-comparison}.
Thus $V_i'>1$ on $(0,b_i^*)$.
The affine extension gives $V_i'=1$ for $x\geq b_i^*$, proving \eqref{eq:vi-regions}.

It remains to prove \eqref{eq:vi-action-inequality}.
Set $\bar\delta=\theta\delta_1+(1-\theta)\delta_2$.
The weighted fixed boundary equations give $(\mathcal LV)_i+f_{i,s}(\cdot,0)=0$ on $(0,b_i^*)$ and, by continuity, also at $b_i^*$.
Differentiating these equations once and taking the left limit at $b_i^*$ gives
$$
q_i\bigl(V_j'(b_i^*)-1\bigr)-\bar\delta
=-a_iV_i'''(b_i^*{-})\leq0,
$$
where the last inequality follows from $V_i''\leq0$ and $V_i''(b_i^*)=0$.
For $x\geq b_i^*$, the affine extensions and the boundary equality give
$$
\begin{aligned}
(\mathcal LV)_i(x)+f_{i,s}(x,0)
&=\int_{b_i^*}^x
\left[q_i\bigl(V_j'(y)-1\bigr)-\bar\delta\right]dy\\
&\leq-a_i(x-b_i^*)V_i'''(b_i^*{-})\leq0,
\end{aligned}
$$
because $V_j'$ is non-increasing on $[0,+\infty)$.
This proves \eqref{eq:vi-action-inequality}.

Combining the waiting region equality, \eqref{eq:vi-regions}, and \eqref{eq:vi-action-inequality} gives
$$
\max\bigl\{(\mathcal LV)_i(x)+f_{i,s}(x,0),\,
1-V_i'(x)\bigr\}=0,\qquad x>0.
$$
Finally, \eqref{eq:special-fixed-boundary} and the ansatz give
$$
\begin{aligned}
&f_{i,s}(x,s)+(\mathcal Lf(\cdot,s))_i(x)=0,
&&0<x\leq b_i^*,\quad s\geq0,\\
&f_{i,x}(x,s)=\beta(s),
&&x\geq b_i^*,\quad s\geq0.
\end{aligned}
$$
At $b_i^*$, $f_{i,xx}$ is evaluated from the waiting region.
Thus \eqref{eq:verification-hjb} holds.
Including the absorbing boundary in the waiting region identifies
the resulting law as $\Xi^{b^*}$.
\end{proof}

The admissibility of $\Xi^{b^*}$ and the remaining conditions of Theorem~\ref{thm:verification} are verified in the next subsection.

\subsection{Admissibility and equilibrium existence}
\label{subsec:admissibility-equilibrium-existence}

Fix a free boundary pair $b^*$ from Theorem~\ref{thm:existence-free-boundaries}, and retain $u_i^{(\ell)}$,$V_i$, and $f_i$ from the preceding subsection.
By \eqref{eq:vi-regions}, the law generated by $V$ is $\widehat\Xi=\Xi^{b^*}$.
For $(x,i)\in E$, write
$$
\widehat\xi=\xi^{x,i,\widehat\Xi},\qquad
\widehat X=X^{x,i,\widehat\Xi},\qquad
\widehat\tau=\tau^{x,i,\widehat\Xi},\qquad
\bar b=\max\{b_1^*,b_2^*\}.
$$

\begin{theorem}[Existence of an equilibrium barrier law]
\label{thm:equilibrium-barrier-existence}
An equilibrium singular control law exists.
More precisely, every solution $b^*$ of \eqref{eq:special-free-boundary} defines an equilibrium barrier law $\widehat\Xi=\Xi^{b^*}$, and
\begin{equation}
\label{eq:equilibrium-barrier-representation}
\begin{aligned}
f_i(x,s)&=\mathbb E_{x,i}\!\left[
\int_{[0,\widehat\tau]}\beta(s+t)\,d\widehat\xi_t
\right],\qquad s\geq0,\\
V_i(x)&=f_i(x,0)=J(x,i;\widehat\xi),
\qquad (x,i)\in E.
\end{aligned}
\end{equation}
\end{theorem}

We verify the remaining conditions of Theorem~\ref{thm:verification} in the following lemma.

\begin{lemma}[Admissibility and the remaining verification conditions]
\label{lem:admissibility-verification-conditions}
The law $\widehat\Xi$ is admissible, and $(V,f)$ satisfies conditions (a) and (c) of Theorem~\ref{thm:verification}, with $p=1$ in
\eqref{eq:verification-growth}.
\end{lemma}

\begin{proof}
The assertions are immediate when $x=0$.
For $x>0$, the initial payment and surplus are
$$
\Delta\widehat\xi_0=(x-b_i^*)^+,\qquad
\widehat X_0=x\wedge b_i^*.
$$
Between successive regime jumps, suppose that the current regime is $k$ and the surplus immediately after any payment at time $r$ is $y\in(0,b_k^*]$.
Until the next regime jump or bankruptcy, set
$$
\begin{aligned}
Y_t&=y+\mu_k(t-r)+\sigma_k(B_t-B_r),\\
L_t&=\sup_{r\leq v\leq t}(Y_v-b_k^*)^+,\\
\widehat X_t&=Y_t-L_t,\qquad
\widehat\xi_t=\widehat\xi_r+L_t.
\end{aligned}
$$
This is the unique Skorokhod reflection at the upper barrier $b_k^*$, stopped when the surplus reaches zero.
At a regime jump $t<\widehat\tau$, use the new regime $\varepsilon_t$ and set
\begin{equation}
\label{eq:admissibility-switch-payment}
\Delta\widehat\xi_t
=(\widehat X_{t-}-b_{\varepsilon_t}^*)^+,\qquad
\widehat X_t=\widehat X_{t-}\wedge b_{\varepsilon_t}^*.
\end{equation}
The regime process has finitely many jumps on every bounded time interval.
Concatenating these reflections and payments therefore gives a unique strong solution up to $\widehat\tau$.
After $\widehat\tau$, keep $\widehat X=0$ and $\widehat\xi$ constant.
The resulting control $\widehat\xi$ is adapted, non-decreasing, and c\`adl\`ag.

By construction,
\begin{equation}
\label{eq:admissibility-state-bound}
0\leq\widehat X_t\leq b_{\varepsilon_t}^*\leq\bar b,
\qquad t\geq0,
\end{equation}
and
$$
\int_0^{t\wedge\widehat\tau}
\mathbf1_{\{\widehat X_r<b_{\varepsilon_r}^*\}}
\,d\widehat\xi_r^c=0.
$$
Each jump is the smallest payment that brings the surplus below or to the current barrier, and $0\leq\Delta\widehat\xi_t\leq\widehat X_{t-}$.
Thus the pair satisfies \eqref{eq:state-reflection}.

To prove that the payoff is finite, put $\rho_T=T\wedge\widehat\tau$.
Integration $d(e^{-\delta_1 t}\widehat X_{t})$ on $[0-,\rho_T]$ gives
$$
\mathbb E_{x,i}\int_{[0,\rho_T]}e^{-\delta_1t}\,d\widehat\xi_t
=x+\mathbb E_{x,i}\int_0^{\rho_T}
e^{-\delta_1t}
\bigl(\mu_{\varepsilon_t}-\delta_1\widehat X_t\bigr)\,dt
-\mathbb E_{x,i}\!\left[
e^{-\delta_1\rho_T}\widehat X_{\rho_T}\right]
\leq x+\frac{\max_{k\in S}\mu_k}{\delta_1}.
$$
The Brownian integral has zero expectation because the volatilities are bounded and $\rho_T\leq T$.
Using $\beta(t)\leq e^{-\delta_1t}$, monotone convergence gives
\begin{equation}
\label{eq:admissibility-payoff-bound}
J(x,i;\widehat\xi)
\leq x+\frac{\max_{k\in S}\mu_k}{\delta_1}<\infty.
\end{equation}
Hence $\widehat\xi\in\mathcal D(x,i)$ for every $(x,i)\in E$, and $\widehat\Xi$ is admissible.

Lemma~\ref{lem:existence-fixed-boundary}, the ansatz, and Theorem~\ref{thm:vi-verification} give the regularity required in (a).
The boundary conditions follow from $u_i^{(\ell)}(0)=0$ and $V_i=f_i(\cdot,0)$.
Furthermore, Lemma~\ref{lem:existence-uniform} gives a constant $C>0$ such that
\begin{equation}
\label{eq:admissibility-candidate-bounds}
\begin{aligned}
|V_i'(x)|+|V_i''(x)|\leq C,\quad
0\leq f_i(x,s)\leq Cxe^{-\delta_1s},\quad
|f_{i,s}(x,s)|\leq Cxe^{-\delta_1s},
\quad x,s\geq0,\; i\in S.
\end{aligned}
\end{equation}
Here $V_i''(0)$ is the right derivative, and derivatives of $u_i^{(\ell)}$ at $b_i^*$ are taken from the left.
In particular, $0\leq V_i(x)\leq Cx$, and thus
$$
|(\mathcal LV)_i(x)|
\leq a_i|V_i''(x)|+\mu_i|V_i'(x)|
+q_i\bigl(|V_j(x)|+|V_i(x)|\bigr)
\leq C(1+x).
$$
Consequently, for every $T>0$,
$$
\max_{i\in S}\left\{
|V_i'(x)|+|(\mathcal LV)_i(x)|
+\sup_{0\leq s\leq T}|f_{i,s}(x,s)|
\right\}\leq C(1+x),\qquad x>0,
$$
which proves \eqref{eq:verification-growth} with $p=1$.

Finally, fix $s\geq0$.
On $\{\widehat\tau\leq T\}$, the terminal value $f_{\varepsilon_{\rho_T}}(\widehat X_{\rho_T},s+\rho_T)$ is zero.
On $\{\widehat\tau>T\}$, $\rho_T=T$, and \eqref{eq:admissibility-state-bound} and \eqref{eq:admissibility-candidate-bounds} give
$$
\mathbb E_{x,i}\!\left[
\left|f_{\varepsilon_{\rho_T}}
(\widehat X_{\rho_T},s+\rho_T)\right|
\right]
\leq C\bar b\,e^{-\delta_1(s+T)}\longrightarrow0.
$$
This proves \eqref{eq:verification-transversality} and completes
the verification of (c).
\end{proof}

\begin{proof}[Proof of Theorem~\ref{thm:equilibrium-barrier-existence}]
Theorem~\ref{thm:existence-free-boundaries} supplies a positive free boundary pair $b^*$.
Then Theorem~\ref{thm:vi-verification} verifies (b) of Theorem~\ref{thm:verification} and identifies the generated law as $\Xi^{b^*}$.
Lemma~\ref{lem:admissibility-verification-conditions} gives admissibility
and conditions (a) and (c).
Therefore, by Theorem~\ref{thm:verification}, $\Xi^{b^*}$ is an equilibrium control law, and it satisfies \eqref{eq:equilibrium-barrier-representation}.
\end{proof}

\section{Numerical analysis of equilibrium barriers}
\label{sec:numerical-illustrations}

We examine the economic implications of the two-regime model in Section~\ref{sec:special-case}. 
We first illustrate the equilibrium dividend law and the value of retained cash, then study shareholder preferences, regime transitions, and operating cash flows.
Unless stated otherwise, parameters not varied in a figure remain at their baseline values. The parameters are illustrative rather than empirical, and all comparative statics concern the displayed ranges.

\subsection{A baseline example}
\label{subsec:numerical-baseline}

Table~\ref{tab:numerical-baseline} gives the baseline parameters and equilibrium barriers. 
Surplus $X$ denotes available cash reserves, dividends record $\xi$ distributions to shareholders, and the regime $\varepsilon$ describes persistent changes in operating conditions. 
This interpretation follows the cash-reserve model of \citet{jiang2012optimal}, where cash flow also follows a regime-modulated Brownian motion.

\begin{table}[htbp]
  \centering
  \caption{Baseline parameters and equilibrium barriers.}
  \label{tab:numerical-baseline}
  \small
  \renewcommand{\arraystretch}{1.15}
  \begin{tabularx}{\linewidth}{@{}ll>{\raggedright\arraybackslash}X@{}}
    \toprule
    Parameter & Baseline value & Economic interpretation \\
    \midrule
    $\mu_1,\mu_2$ & $0.12,\ 0.06$
      & Mean net cash inflow in each regime \\
    $\sigma_1,\sigma_2$ & $0.25,\ 0.25$
      & Cash-flow volatility in each regime \\
    $q_1,q_2$ & $0.35,\ 0.25$
      & Transition intensity out of each regime \\
    $\delta_1,\delta_2$ & $0.04,\ 0.12$
      & Discount rates of the two shareholder groups \\
    $\theta$ & $0.55$
      & Evaluation weight of the more patient shareholders \\
    \midrule
    $\bar\delta$ & $0.076$
      & Initial instantaneous discount rate \\
    $\kappa,p$ & $0.60,\ 0.4167$
      & Switching speed and stationary probability of regime 1 \\
    $b_1^*,b_2^*$ & $0.806783,\ 0.858651$
      & Equilibrium dividend barriers \\
    \bottomrule
  \end{tabularx}
  \par\smallskip
  \begin{minipage}{\linewidth}
    \footnotesize
    \textit{Notes:} Time is measured in years, and cash is expressed in a
    normalized monetary unit. The derived quantities are
    $\bar\delta=\theta\delta_1+(1-\theta)\delta_2$,
    $\kappa=q_1+q_2$, and $p=q_2/(q_1+q_2)$.
  \end{minipage}
\end{table}

Regime 1 has twice the mean net cash inflow of regime 2, with the same volatility. 
Their expected durations are $1/q_1\approx2.86$ and $1/q_2=4$
years. 
We observe that the barrier of the less profitable and more persistent regime 2 is higher:
$$
b_1^*=0.806783<b_2^*=0.858651.
$$
For reserves between these barriers, the firm retains cash in regime 2 and pays the excess over $b_1^*$ in regime 1. 
The ordering may reflect the combined effects of several factors rather than a single factor, which will be demonstrated in the numerical experiments presented later. 
\begin{figure}[htbp]
  \centering

  \def\PanelAFile{surplus_path.png}
  \def\PanelBFile{dividends_path.png}

  \captionsetup[subfigure]{
    font=small,
    labelfont=normalfont,
    textfont=normalfont,
    justification=centering
  }

  \begin{subfigure}[t]{0.98\linewidth}
    \centering
    \IfFileExists{\PanelAFile}{%
      \includegraphics[width=\linewidth]{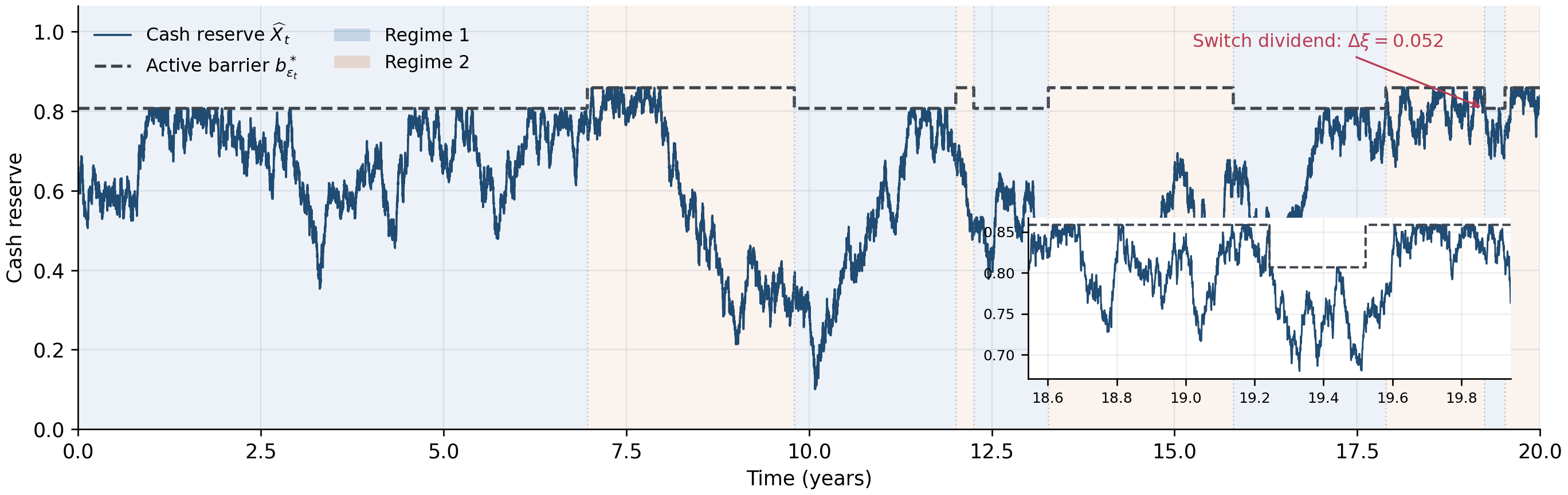}%
    }{%
      \fbox{\parbox[c][0.20\textheight][c]{0.92\linewidth}{%
        \centering First panel
      }}%
    }
    \caption{Cash reserves and the active barrier.}
    \label{fig:numerical-sample-path-reserves}
  \end{subfigure}

  \par\medskip

  \begin{subfigure}[t]{0.98\linewidth}
    \centering
    \IfFileExists{\PanelBFile}{%
      \includegraphics[width=\linewidth]{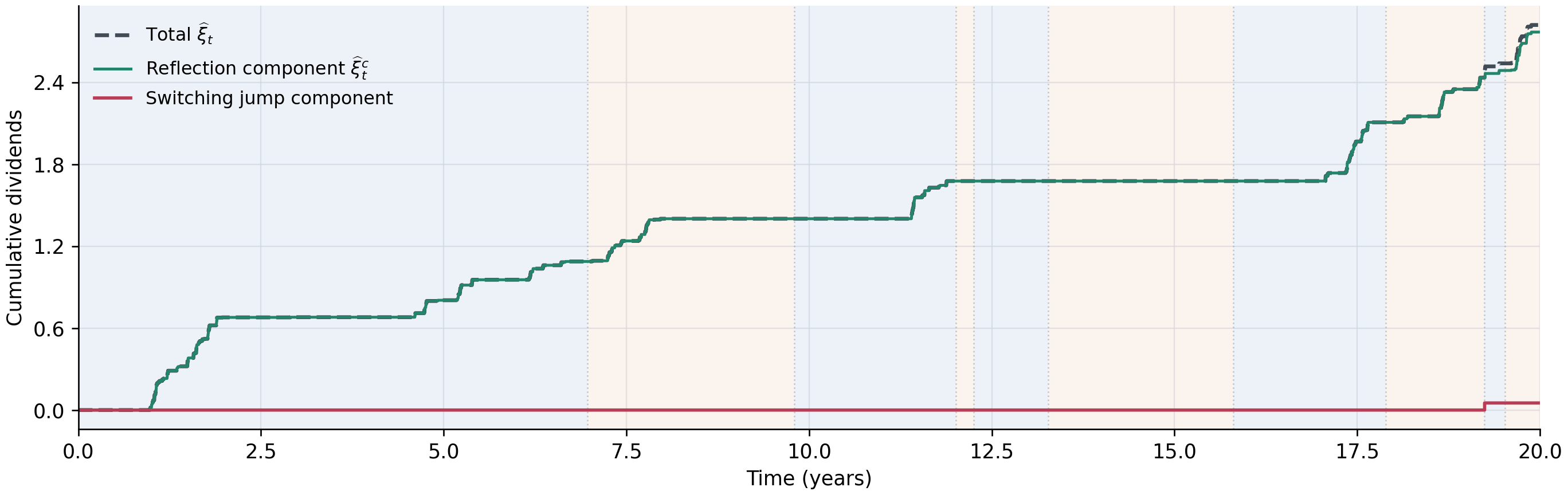}%
    }{%
      \fbox{\parbox[c][0.20\textheight][c]{0.92\linewidth}{%
        \centering Second panel
      }}%
    }
    \caption{Cumulative dividends.}
    \label{fig:numerical-sample-path-dividends}
  \end{subfigure}

  \caption{An illustrative controlled sample path.}
  \label{fig:numerical-sample-path}
\end{figure}

Figure~\ref{fig:numerical-sample-path} shows the equilibrium singular control law along a twenty-year sample path. 
In Figure~\ref{fig:numerical-sample-path}\subref{fig:numerical-sample-path-reserves}, the firm retains cash below the barrier and distributes further accumulation through downward reflection at the barrier. 
The continuous dividend component satisfies
$$
\int_0^{t\wedge\widehat\tau}
\mathbf 1_{\{\widehat X_s<b_{\varepsilon_s}^*\}}
\,d\widehat\xi_s^c=0.
$$
Comparing Figure~\ref{fig:numerical-sample-path}\subref{fig:numerical-sample-path-reserves} and \subref{fig:numerical-sample-path-dividends} makes the reflection mechanism visible. 
Continuous singular payments increase at the barrier and remain unchanged while cash is strictly below it.

A regime switch before bankruptcy can also trigger a lump-sum dividend:
$$
\Delta\widehat\xi_t
=\bigl(\widehat X_{t-}-b_{\varepsilon_t}^*\bigr)^+.
$$
The annotated switch near year~19 lowers the barrier from $b_2^*$ to $b_1^*$, producing a dividend of approximately $0.052$ and a corresponding downward jump in cash. 
Figure~\ref{fig:numerical-sample-path}\subref{fig:numerical-sample-path-dividends} separates this payment from continuous reflection payments. 
An upward barrier change always allows the cash retention;
a downward change produces a dividend payment only when cash exceeds the new barrier. 
Such switching payments also arise in the time-consistent model
of \citet{jiang2012optimal}. 

\begin{figure}[htbp]
  \centering

  \def\PanelAFile{equilibrium_value.png}
  \def\PanelBFile{marginal_value_of_cash.png}
  \def\PanelCFile{continuation_premium.png}

  \captionsetup[subfigure]{
    font=small,
    labelfont=normalfont,
    textfont=normalfont,
    justification=centering
  }

  \begin{subfigure}[t]{0.32\linewidth}
    \centering
    \IfFileExists{\PanelAFile}{%
      \includegraphics[width=\linewidth]{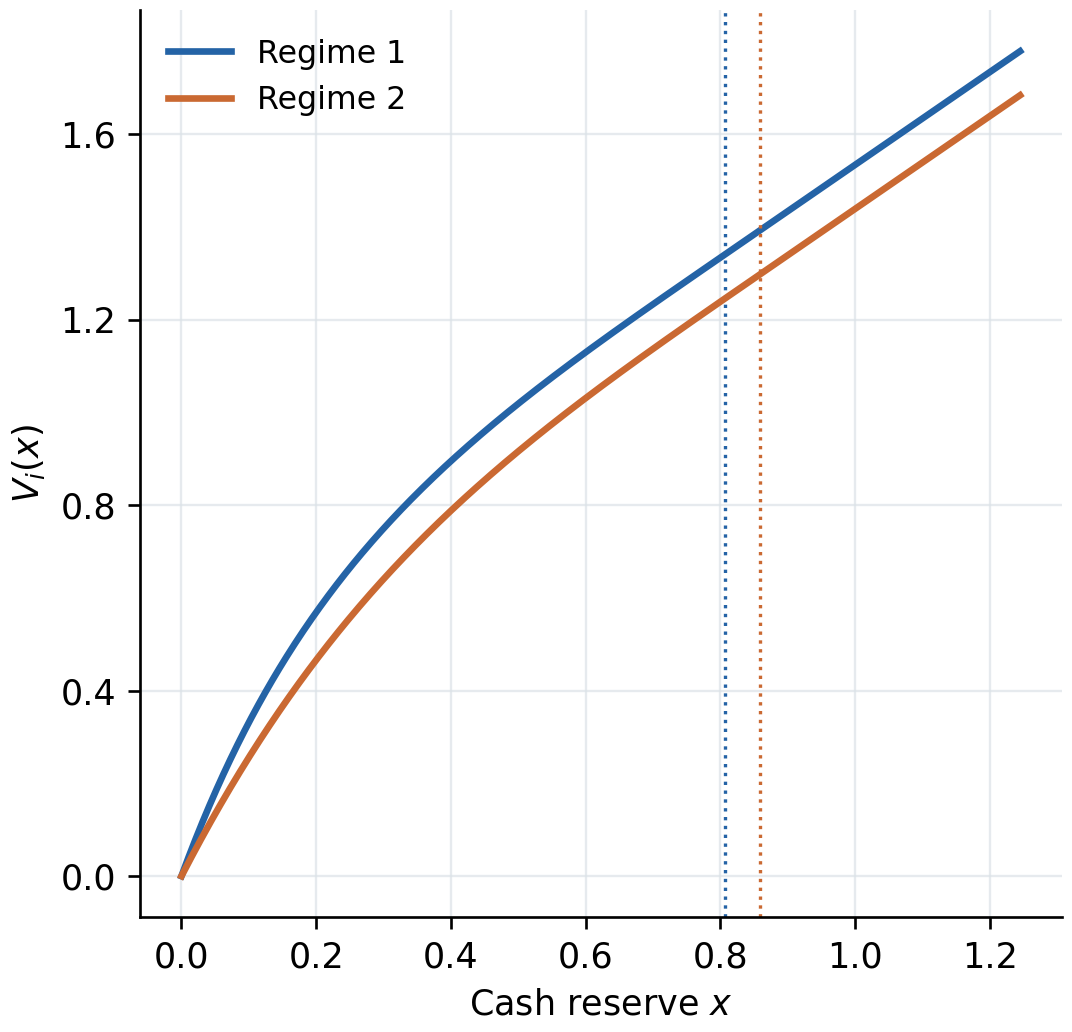}%
    }{%
      \fbox{\parbox[c][0.20\textheight][c]{0.92\linewidth}{%
        \centering First panel
      }}%
    }
    \caption{Equilibrium values $V_i(x)$.}
    \label{fig:numerical-value-functions-value}
  \end{subfigure}\hfill
  \begin{subfigure}[t]{0.32\linewidth}
    \centering
    \IfFileExists{\PanelBFile}{%
      \includegraphics[width=\linewidth]{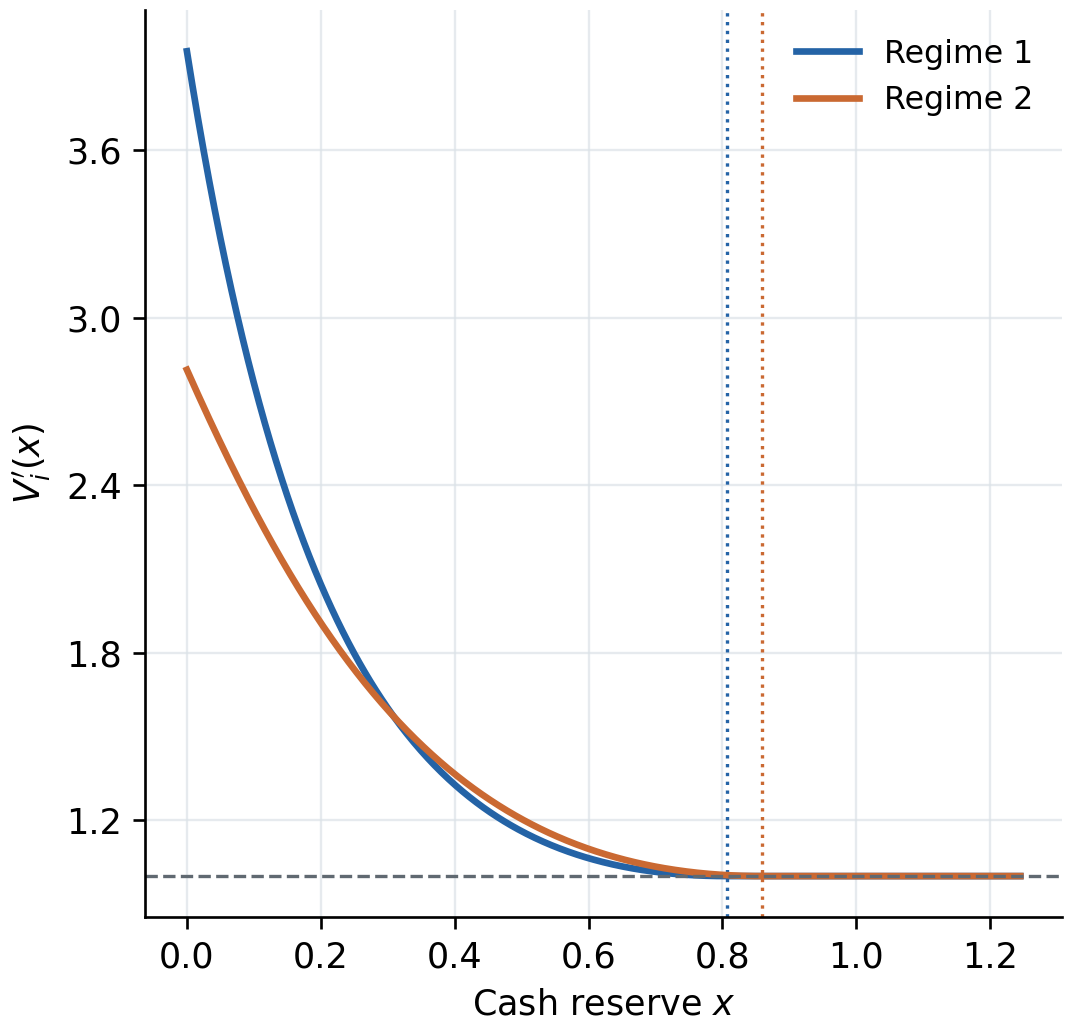}%
    }{%
      \fbox{\parbox[c][0.20\textheight][c]{0.92\linewidth}{%
        \centering Second panel
      }}%
    }
    \caption{Marginal values $V_i'(x)$.}
    \label{fig:numerical-value-functions-marginal}
  \end{subfigure}\hfill
  \begin{subfigure}[t]{0.32\linewidth}
    \centering
    \IfFileExists{\PanelCFile}{%
      \includegraphics[width=\linewidth]{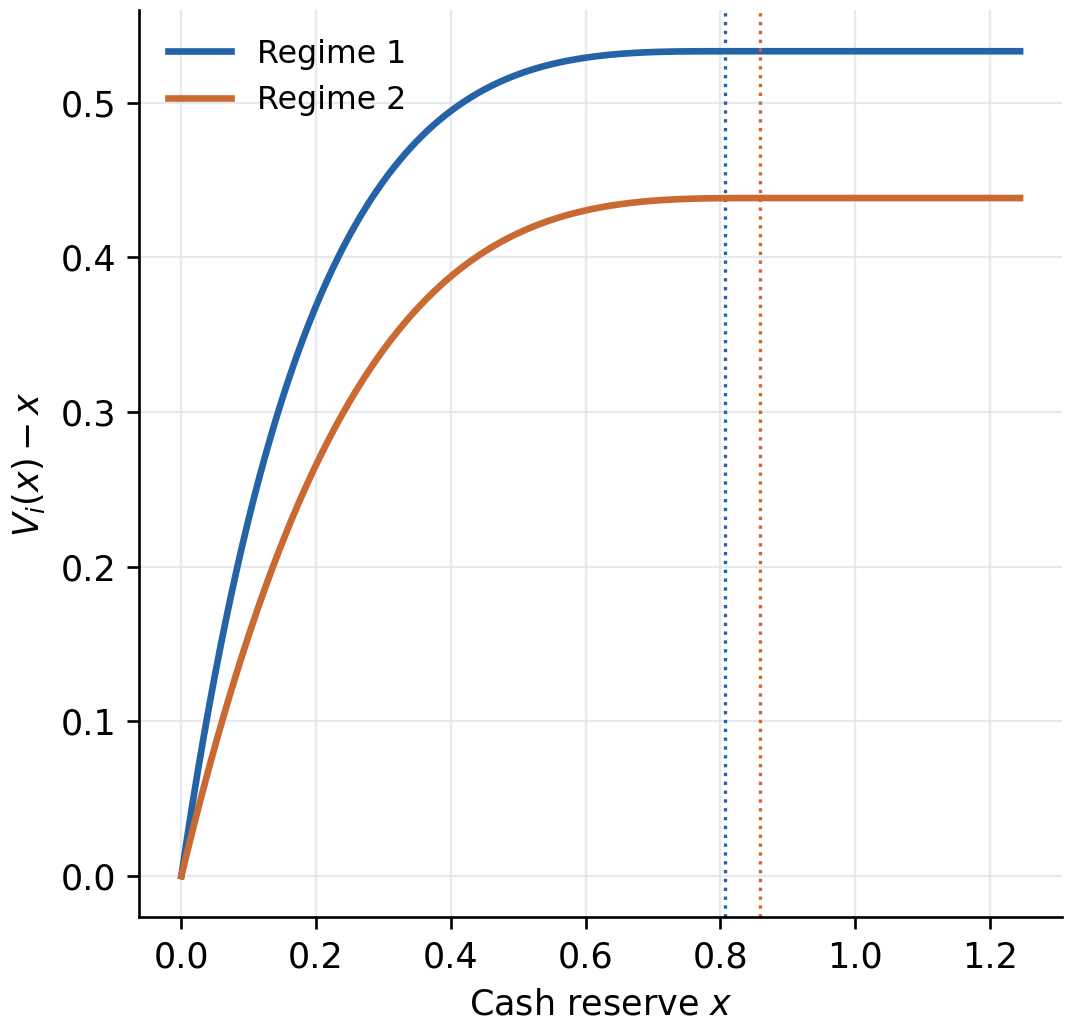}%
    }{%
      \fbox{\parbox[c][0.20\textheight][c]{0.92\linewidth}{%
        \centering Third panel
      }}%
    }
    \caption{Continuation premiums $V_i(x)-x$.}
    \label{fig:numerical-value-functions-premium}
  \end{subfigure}

  \caption{Equilibrium values.
    The subfigures show $V_i(x)$, $V_i'(x)$, and $V_i(x)-x$.
    Vertical dotted lines identify the corresponding barriers.}
  \label{fig:numerical-value-functions}
\end{figure}

Figure~\ref{fig:numerical-value-functions}\subref{fig:numerical-value-functions-value} shows that the equilibrium value increases with initial cash and is higher in regime 1 at the displayed ranges. In Figure~\ref{fig:numerical-value-functions}\subref{fig:numerical-value-functions-marginal}, the marginal value decreases toward one:
$$
V_i'(x)>1\quad\text{for }0<x<b_i^*,
\qquad
V_i'(x)=1\quad\text{for }x\geq b_i^*.
$$
Under the equilibrium law, retaining an additional unit of cash below the barrier gains more than one unit of expected discounted dividends.
At the barrier, its marginal continuation value equals its immediate payout value. 
\citet{bolton2011unified} similarly show that liquidity can preserve operating opportunities and have a marginal value exceeding one. 
Their model also includes investment, financing, and risk management while ours focuses on dividends.

Figure~\ref{fig:numerical-value-functions}\subref{fig:numerical-value-functions-premium} shows the continuation premium $V_i(x)-x$ relative to distributing all initial cash as dividends and subsequently going bankrupt. 
It increases below the barrier and is constant above it,  as $V_i(x)-x=V_i(b_i^*)-b_i^*,\qquad x\geq b_i^*$.
The plateaus are approximately $0.53$ and $0.44$ in regimes 1 and 2.

\subsection{Shareholder patience and preference heterogeneity}
\label{subsec:numerical-shareholder-preferences}

The discount function
$$
\beta(s)=\theta e^{-\delta_1s}+(1-\theta)e^{-\delta_2s}
$$
aggregates two shareholder groups' preferences with different patience.
Using $\delta_1<\delta_2$, the first group is more patient, and $\theta$ is its weight in management's evaluation.
\citet{zhao2014dividend} obtain mixtures of exponential discounting by aggregating shareholders' returns with weights equal to their shares. 
Therefore, $\theta$ can be explained as the shareholding proportion of the more patient shareholder group.
\citet{ebert2020weighted} provide an economic foundation for weighted discounting and diversity in time preferences.

\begin{figure}[htbp]
  \centering
  \IfFileExists{Shareholder_weights_and_dividend_barriers.png}{%
    \includegraphics[width=0.7\linewidth]{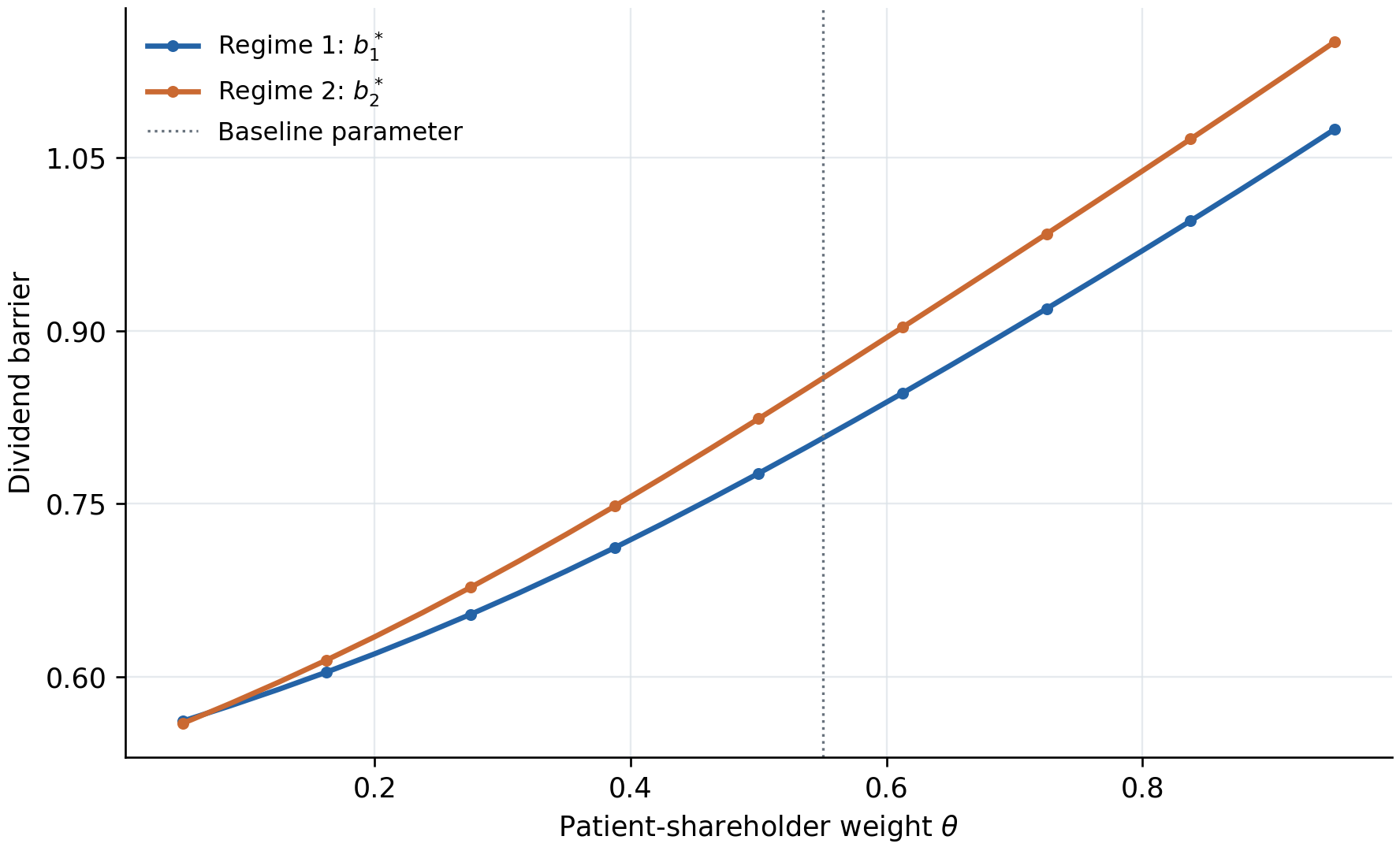}%
  }{%
    \fbox{\parbox[c][0.20\textheight][c]{0.92\linewidth}{%
      \centering
      Shareholder weights and dividend barriers\\[0.5em]
      \texttt{\detokenize{Shareholder_weights_and_dividend_barriers.png}}%
    }}%
  }
  \caption{Shareholder weights and dividend barriers. The discount rates and operating parameters remain fixed as $\theta$ varies. The dotted line marks the baseline weight $\theta=0.55$.}
  \label{fig:numerical-shareholder-weights}
\end{figure}

Figure~\ref{fig:numerical-shareholder-weights} shows that both barriers increase with the weight of patient shareholders, from approximately $0.56$ in both regimes to $1.07$ and $1.15$, respectively. 
A higher weight on future dividends expands the waiting regions and increases cash retention. \citet[Example~4.4]{zhao2014dividend} report a similar response: a greater weight on the lower discount rate raises the dividend threshold. 
Their dividend rate is bounded, whereas our model allows singular payments.

Changing $\theta$ also changes initial average impatience. 
To separate preference dispersion, Figure~\ref{fig:numerical-preference-heterogeneity} fixes $\theta=0.55$ and $\bar\delta=-\beta'(0)=\theta\delta_1+(1-\theta)\delta_2=0.076$ and varies $d=\delta_2-\delta_1$ through
\begin{equation}
\label{eq:numerical-rate-spread}
\delta_1=\bar\delta-(1-\theta)d,
\qquad
\delta_2=\bar\delta+\theta d,
\qquad
0\leq d<\frac{\bar\delta}{1-\theta}.
\end{equation}
The mean discount rate remains fixed, while its variance is
$\theta(1-\theta)d^2$.
This is the mean-preserving spread interpretation of greater group diversity in \citet{rothschild1978increasing} and \citet{ebert2020weighted}.
\begin{figure}[htbp]
  \centering

  \def\PanelAFile{Equal_initial_discount_rates.png}
  \def\PanelBFile{Barriers_under_fixed_initial_discounting.png}

  \captionsetup[subfigure]{
    font=small,
    labelfont=normalfont,
    textfont=normalfont,
    justification=centering
  }

  \begin{subfigure}[t]{0.485\linewidth}
    \centering
    \IfFileExists{\PanelAFile}{%
      \includegraphics[width=\linewidth]{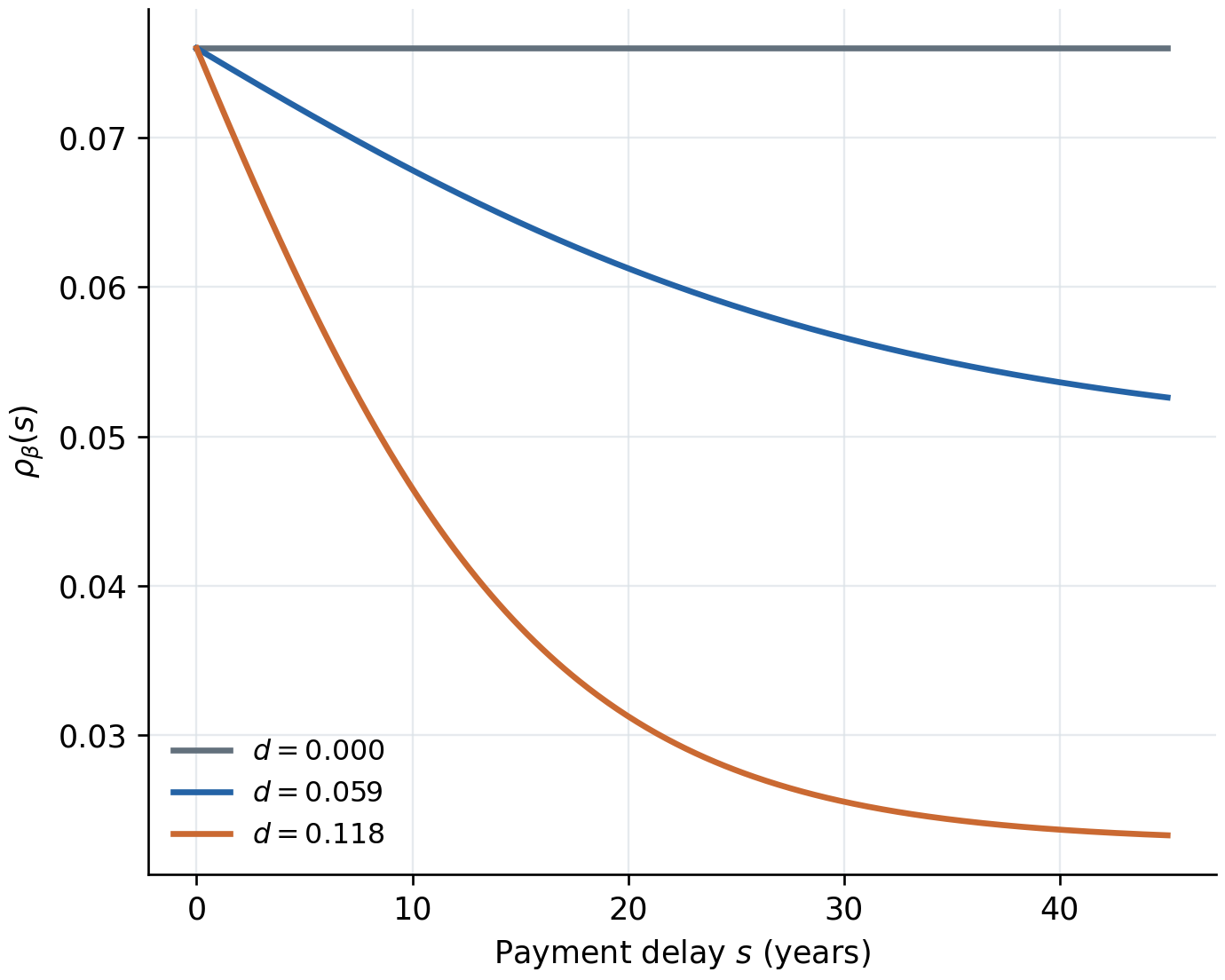}%
    }{%
      \fbox{\parbox[c][0.20\textheight][c]{0.92\linewidth}{%
        \centering First panel
      }}%
    }
    \caption{Instantaneous discount rates $\rho_\beta(s)$.}
    \label{fig:numerical-preference-heterogeneity-rates}
  \end{subfigure}\hfill
  \begin{subfigure}[t]{0.485\linewidth}
    \centering
    \IfFileExists{\PanelBFile}{%
      \includegraphics[width=\linewidth]{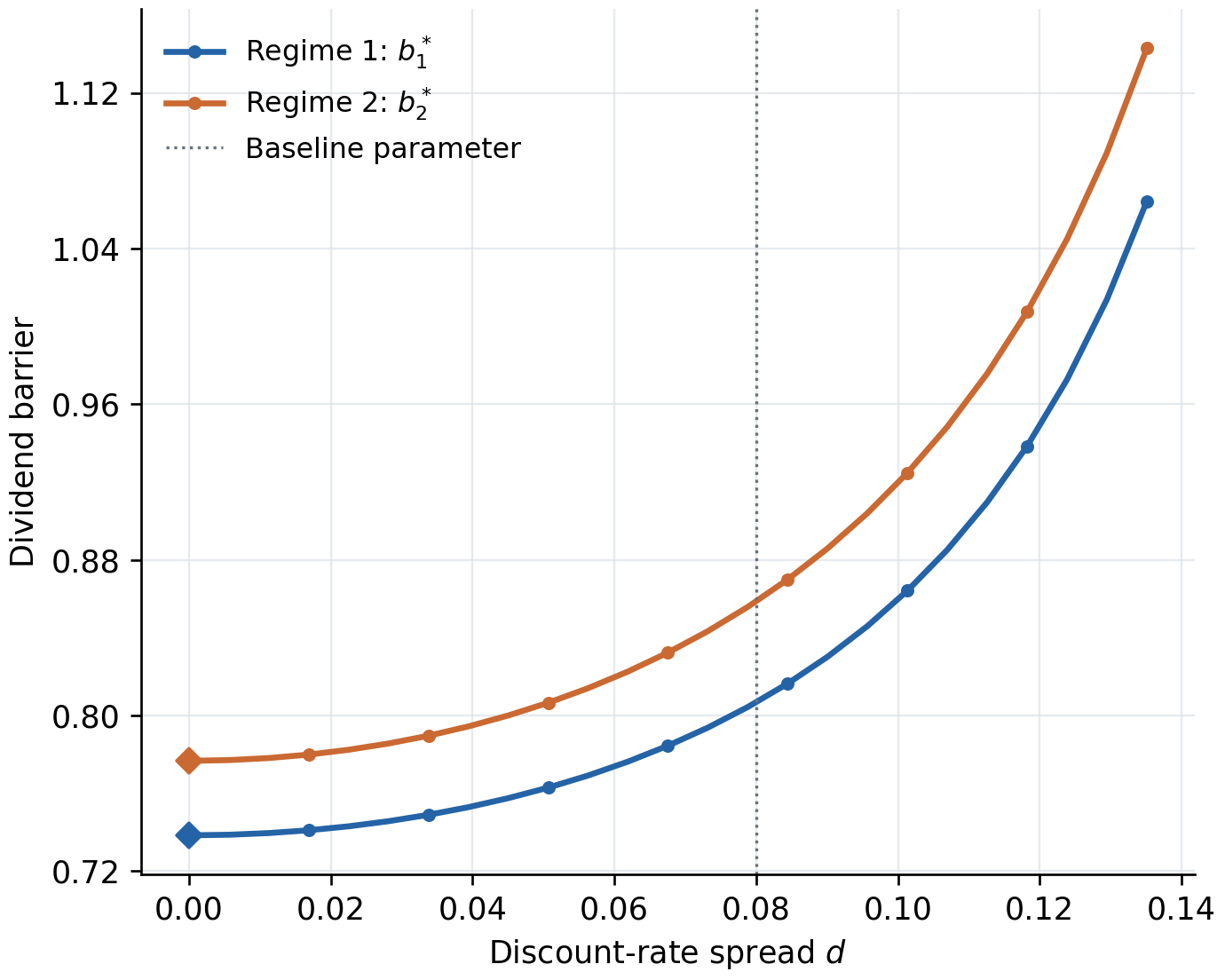}%
    }{%
      \fbox{\parbox[c][0.20\textheight][c]{0.92\linewidth}{%
        \centering Second panel
      }}%
    }
    \caption{Equilibrium dividend barriers.}
    \label{fig:numerical-preference-heterogeneity-barriers}
  \end{subfigure}

  \caption{Preference heterogeneity at a fixed initial discount rate.
    Subfigure (a) compares instantaneous discount-rate profiles.
    Subfigure (b) shows barriers as the spread varies; diamonds mark the exponential benchmark $d=0$, and the dotted line identifies
    the baseline spread $d=0.08$.}
  \label{fig:numerical-preference-heterogeneity}
\end{figure}

Figure~\ref{fig:numerical-preference-heterogeneity}\subref{fig:numerical-preference-heterogeneity-rates} plots the instantaneous discount rate
$$
\rho_\beta(s)
=-\frac{\beta'(s)}{\beta(s)}
=\frac{\theta\delta_1e^{-\delta_1s}
+(1-\theta)\delta_2e^{-\delta_2s}}
{\theta e^{-\delta_1s}+(1-\theta)e^{-\delta_2s}}.
$$
All curves start at $\bar\delta$.
The rate is constant when $d=0$ and declines toward $\delta_1$ when $d>0$, because the patient component gains weight at longer delays.
In particular,
$$
\rho_\beta(0)=\bar\delta,
\qquad
-\rho_\beta'(0)=\theta(1-\theta)d^2.
$$
With $\theta$ and $\bar\delta$ fixed, a larger $d$ also gives a lower $\rho_\beta(s)$ at every fixed delay $s>0$.
Convexity of the exponential in the discount rate gives $\beta(s)\geq e^{-\bar\delta s}$, strictly for $s>0$ and $d>0$.
Thus, equal initial impatience can coexist with different valuations of later dividends.

Figure~\ref{fig:numerical-preference-heterogeneity}\subref{fig:numerical-preference-heterogeneity-barriers} shows that both barriers increase with the spread.
The exponential benchmark $d=0$ gives barriers of approximately  $(0.738,0.777)$, compared with $(0.807,0.859)$ at the baseline spread $d=0.08$, increases of about $9\%$ and $11\%$.
This increase is intuitively consistent with Figure~\ref{fig:numerical-preference-heterogeneity}\subref{fig:numerical-preference-heterogeneity-rates}: a larger $d$ lowers instantaneous discount rates at positive delays and places greater weight on later dividends, making cash retention to sustain operations more attractive.
Using the mean discount rate would therefore understate both barriers for these parameters.

\subsection{Regime persistence and transition expectations}
\label{subsec:numerical-regime-transitions}

We separate switching speed from stationary regime composition by writing
$$
q_1=\kappa(1-p),\qquad q_2=\kappa p,
$$
where $\kappa=q_1+q_2$ and $p=q_2/(q_1+q_2)$.
The stationary distribution of regimes is $(p,1-p)$. 
At fixed $p$, increasing $\kappa$ shortens both regime durations. 
At fixed $\kappa$, increasing $p$ lowers the departure intensity from regime 1 and raises the transition intensity from regime 2 into regime 1.

\begin{figure}[htbp]
  \centering

  \def\PanelAFile{Fixed_stationary_regime_composition.png}
  \def\PanelBFile{Fixed_switching_speed.png}

  \captionsetup[subfigure]{
    font=small,
    labelfont=normalfont,
    textfont=normalfont,
    justification=centering
  }

  \begin{subfigure}[t]{0.485\linewidth}
    \centering
    \IfFileExists{\PanelAFile}{%
      \includegraphics[width=\linewidth]{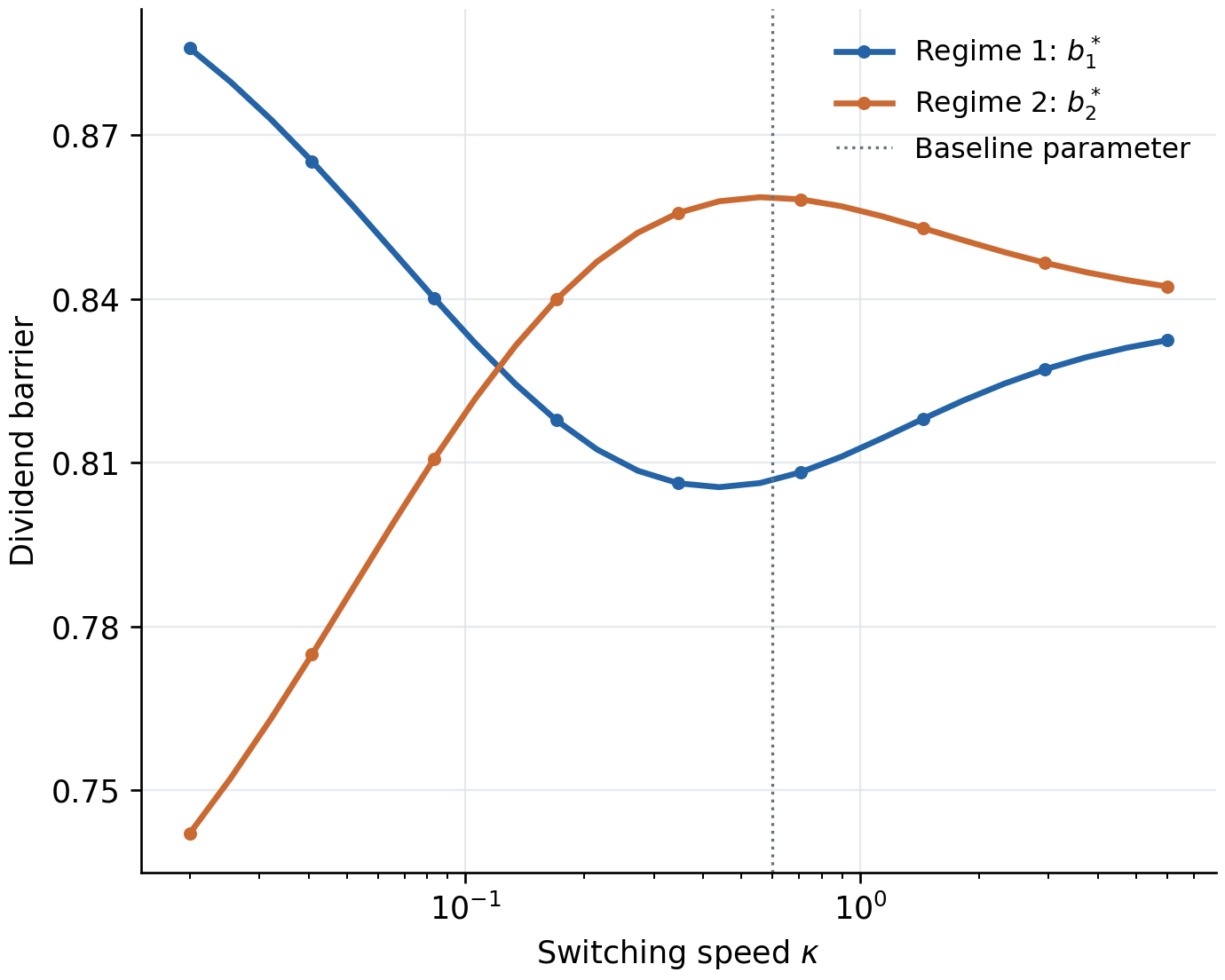}%
    }{%
      \fbox{\parbox[c][0.20\textheight][c]{0.92\linewidth}{%
        \centering First panel
      }}%
    }
    \caption{Switching speed $\kappa$.}
    \label{fig:numerical-regime-transitions-speed}
  \end{subfigure}\hfill
  \begin{subfigure}[t]{0.485\linewidth}
    \centering
    \IfFileExists{\PanelBFile}{%
      \includegraphics[width=\linewidth]{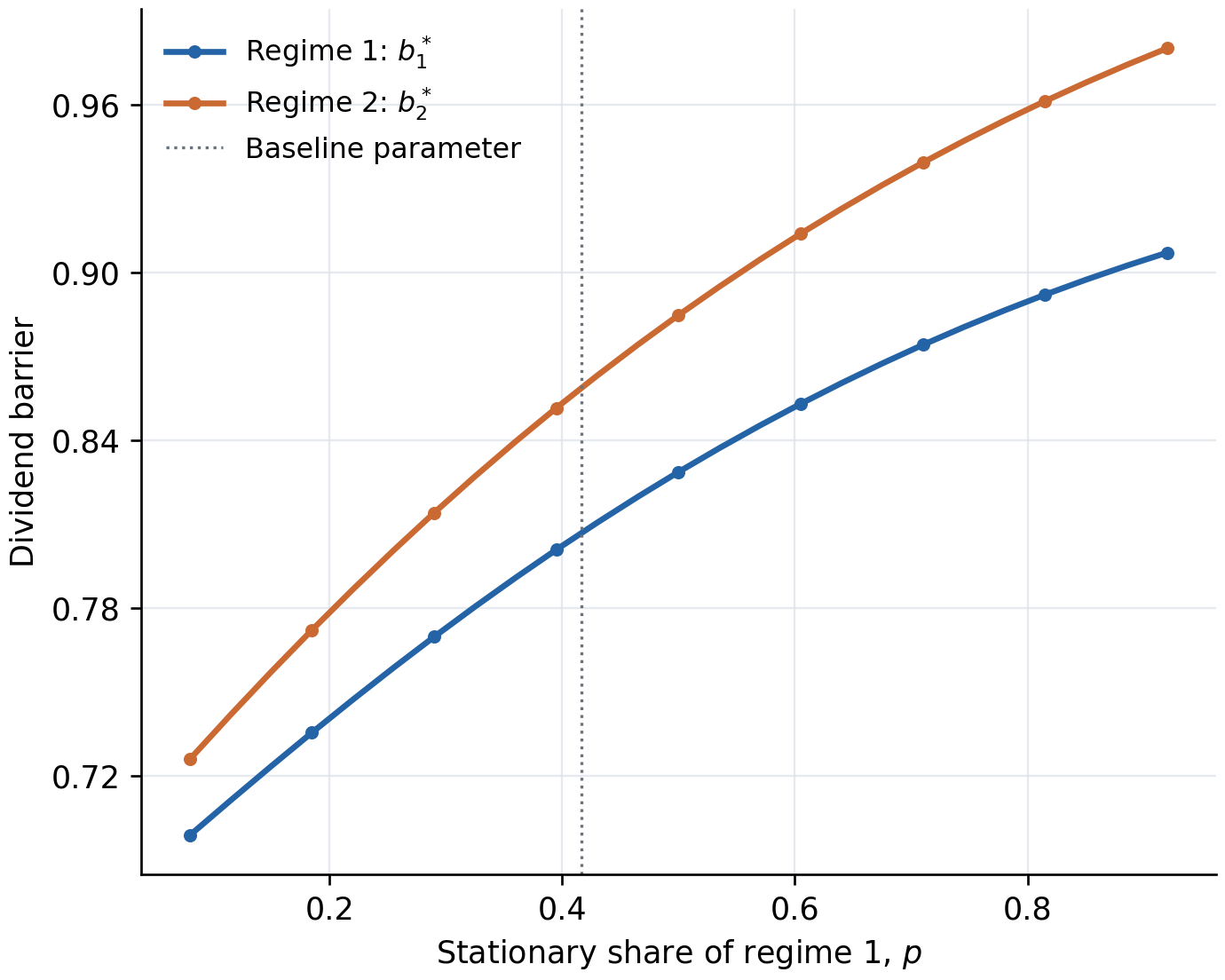}%
    }{%
      \fbox{\parbox[c][0.20\textheight][c]{0.92\linewidth}{%
        \centering Second panel
      }}%
    }
    \caption{Stationary probability $p$.}
    \label{fig:numerical-regime-transitions-probability}
  \end{subfigure}

  \caption{Regime transitions and dividend barriers.
    Subfigure (a) varies switching speed at fixed stationary composition.
    Subfigure (b) varies the stationary probability of regime 1 at fixed
    switching speed. Dotted lines mark the baseline parameters.}
  \label{fig:numerical-regime-transitions}
\end{figure}

Figure~\ref{fig:numerical-regime-transitions}\subref{fig:numerical-regime-transitions-speed} shows non-monotone responses to switching speed. The regime-1 barrier first decreases and then increases; the regime-2 barrier first increases and then decreases.
Their ordering reverses: $b_1^*>b_2^*$ under very slow switching, but $b_1^*<b_2^*$ at the baseline. 
At the largest displayed speed, the barriers are close, at approximately $0.83$ and $0.84$. 
This suggests smaller policy differences under faster switching, which is consistent with the economic intuition that fast switching shortens the expected stay in either regime, reducing the influence of the initial regime on cash retention.

The initial decline in $b_1^*$ and rise in $b_2^*$ may reflect the value of preserving profitable operating opportunities.
Faster switching shortens the firm's stay in regime 1, making cash retention to benefit from the higher profitability of regime 1 less attractive.
In regime 2, earlier access to regime 1 encourages the firm to retain cash and avoid bankruptcy before that opportunity arrives.
The later reversal may reflect precautionary cash retention.
In regime 1, the firm anticipates a more imminent transition to the less profitable regime 2.
It may therefore retain more cash to respond to lower inflows after the switch, protecting subsequent operating opportunities and raising $b_1^*$.
In regime 2, low inflows become more temporary as the firm returns more quickly to regime 1, which may reduce the need for a large reserve and lower $b_2^*$.
These findings complement \citet{jiang2012optimal}.

Figure~\ref{fig:numerical-regime-transitions}\subref{fig:numerical-regime-transitions-probability} shows that both barriers increase with $p$. 
At fixed $\kappa$, a higher $p$ prolongs regime 1 and shortens the wait to enter it from regime 2. 
Using $\mu_1>\mu_2$, cash-flow prospects improve from either initial regime. 
The higher barriers suggest that better prospects encourage cash retention to support future operations and dividends.

\begin{figure}[htbp]
  \centering
  \IfFileExists{Preference_heterogeneity_and_regime_switching.png}{%
    \includegraphics[width=0.98\linewidth]{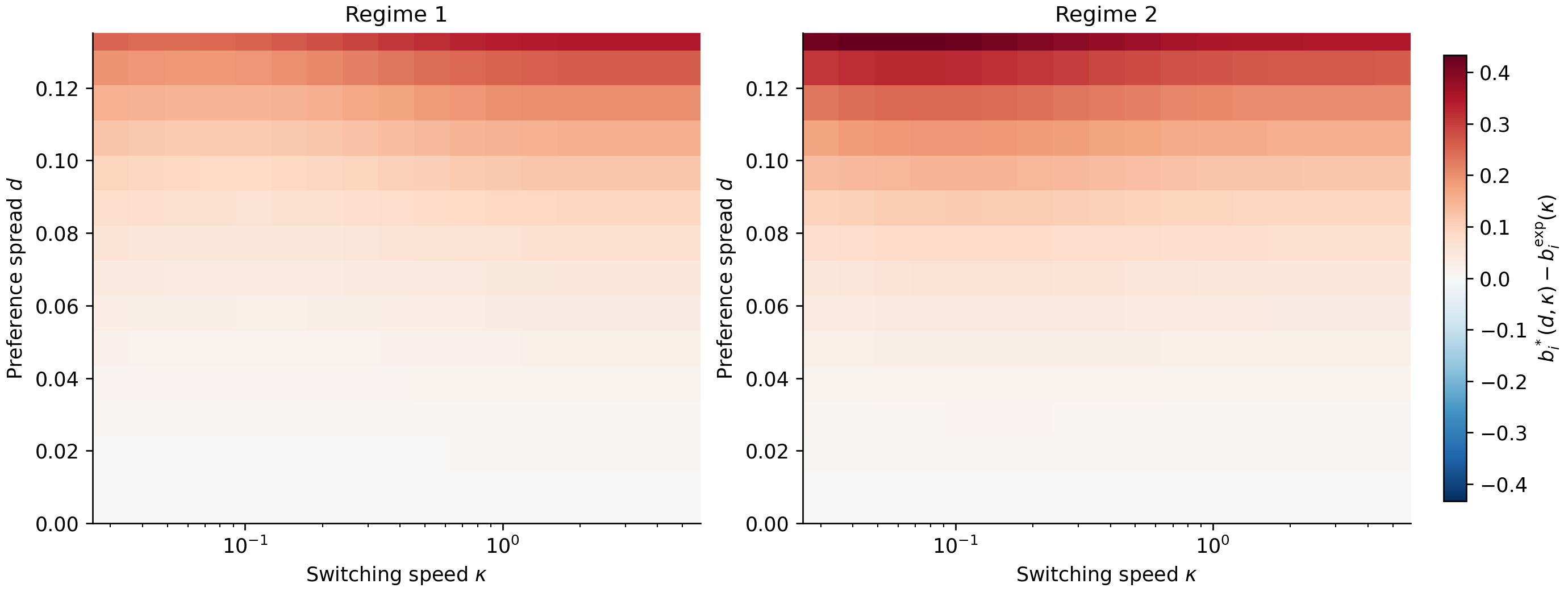}%
  }{%
    \fbox{\parbox[c][0.20\textheight][c]{0.92\linewidth}{%
      \centering
      Preference heterogeneity and regime switching\\[0.5em]
      \texttt{\detokenize{Preference_heterogeneity_and_regime_switching.png}}%
    }}%
  }
  \caption{Preference heterogeneity and regime switching. The heatmaps report the barrier differences relative to exponential discounting over discount-rate spreads and switching speeds. Both panels use the same color scale.}
  \label{fig:numerical-interaction}
\end{figure}

Figure~\ref{fig:numerical-interaction} combines preference dispersion and switching speed. 
The panels show
$$
\Delta b_i(d,\kappa)
=b_i^*(d,\kappa)-b_i^{\mathrm{exp}}(\kappa),
$$
where the exponential benchmark has the same cash-flow parameters and transition intensities, with discount rate $\bar\delta=0.076$.
The discount rates follow \eqref{eq:numerical-rate-spread}, while $\theta$ and $p$ remain at their baseline values.
The differences are zero at $d=0$ and positive at the displayed points with $d>0$.

Greater dispersion $d$ increases these barrier differences. 
At the largest spread, they are approximately $0.25$ and $0.42$ in regimes~1 and~2 under slow switching, compared with approximately $0.35$ in both regimes under faster switching. 
Faster switching thus increases the effect in regime 1 and reduces it in regime 2, bringing the two responses closer together.

Under slow switching, longer stays make future cash flows more
dependent on the initial regime, as such  preference dispersion can affect the barriers differently. 
Faster switching weakens this dependence and gives more similar operating prospects from either initial regime.
This may explain the closer responses. 
Additionally, as in Figure~\ref{fig:numerical-preference-heterogeneity}, dispersion raises barriers relative to exponential discounting even when initial average impatience is fixed.

\subsection{Cash-flow profitability and volatility}
\label{subsec:numerical-cash-flow}

Figures~\ref{fig:numerical-profitability} and~\ref{fig:numerical-volatility} vary mean net cash inflows and cash-flow volatility, respectively. 
A change in one regime can affect both barriers because payout decisions also depend on operating conditions after a regime transition.

\begin{figure}[htbp]
  \centering
  \captionsetup[subfigure]{font=small,labelfont=normalfont,textfont=normalfont,justification=centering}
  \begin{subfigure}[t]{0.485\linewidth}
    \centering
    \IfFileExists{Regime1_profitability.png}{%
      \includegraphics[width=0.98\linewidth]{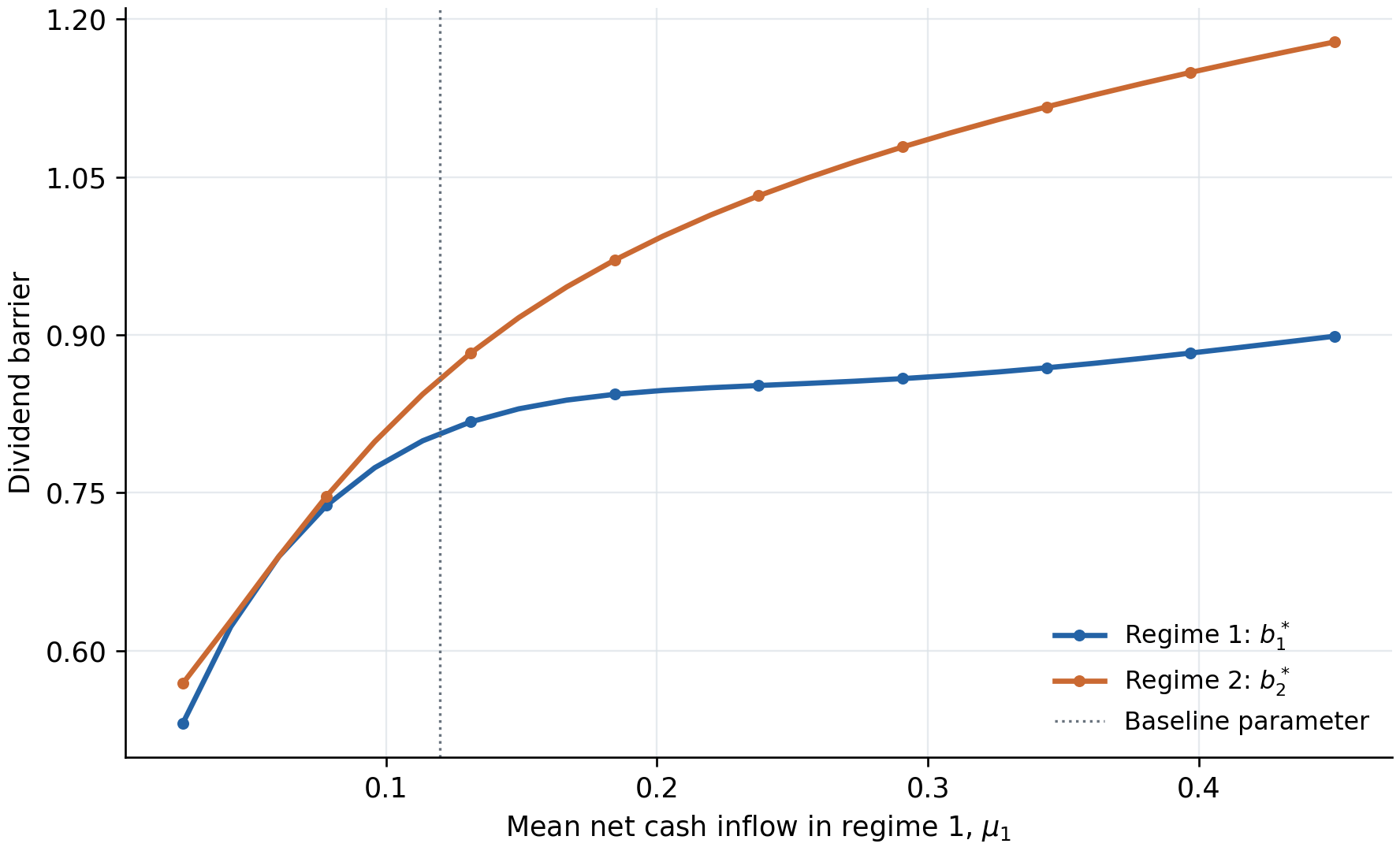}%
    }{%
      \fbox{\parbox[c][0.20\textheight][c]{0.92\linewidth}{%
        \centering
        Profitability and dividend barriers\\[0.5em]
        \texttt{\detokenize{Regime1_profitability.png}}%
      }}%
    }
    \caption{Regime 1: varying $\mu_1$.}
    \label{fig:numerical-profitability-mu1}
  \end{subfigure}\hfill
  \begin{subfigure}[t]{0.485\linewidth}
    \centering
    \IfFileExists{Regime2_profitability.png}{%
      \includegraphics[width=0.98\linewidth]{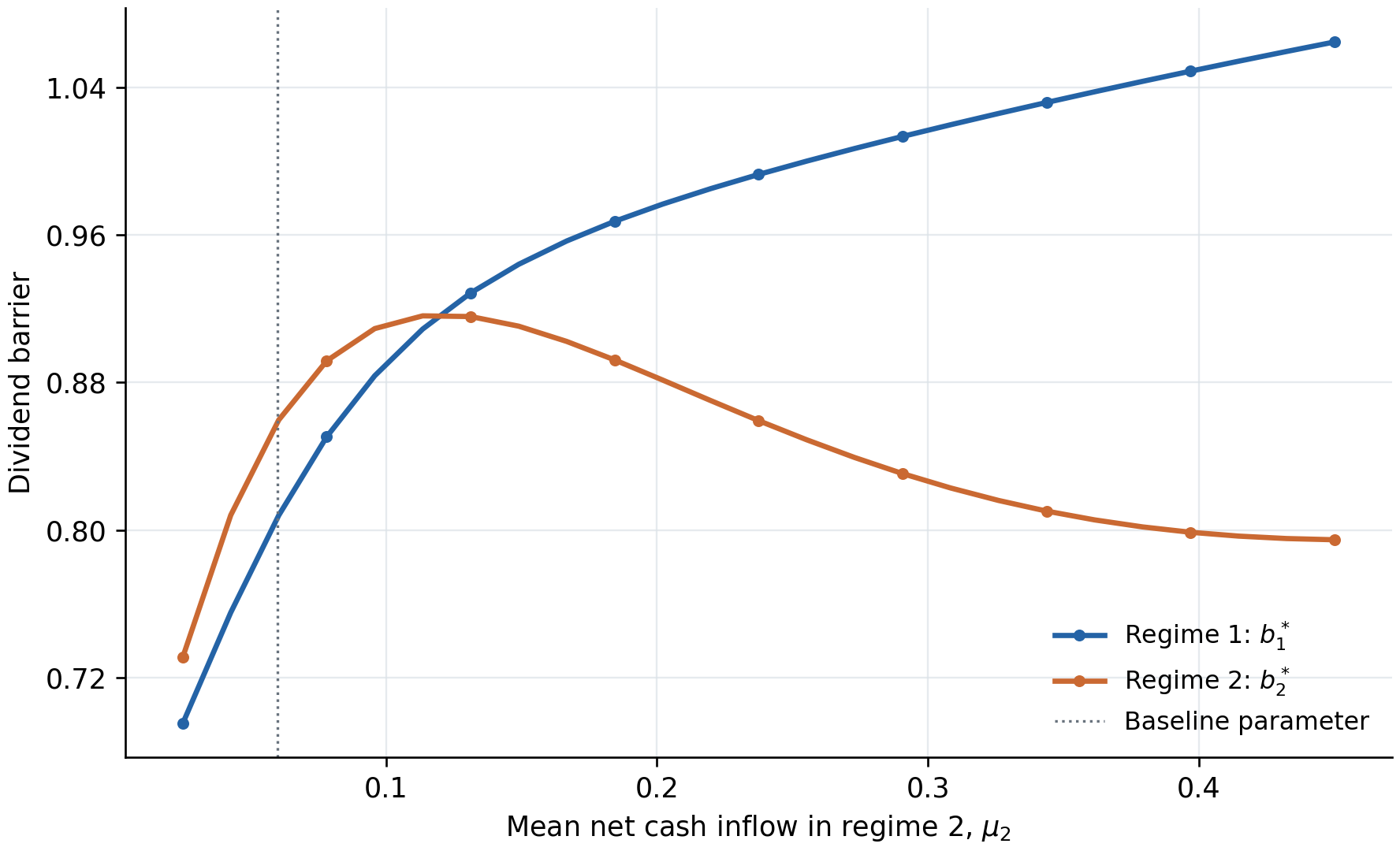}%
    }{%
      \fbox{\parbox[c][0.20\textheight][c]{0.92\linewidth}{%
        \centering
        Profitability and dividend barriers\\[0.5em]
        \texttt{\detokenize{Regime2_profitability.png}}%
      }}%
    }
    \caption{Regime 2: varying $\mu_2$.}
    \label{fig:numerical-profitability-mu2}
  \end{subfigure}
  \caption{Profitability and equilibrium dividend barriers.
    Subfigure (a) varies $\mu_1$, and subfigure (b) varies $\mu_2$.
    All remaining parameters are held at their baseline values.
    Vertical dotted lines mark the baseline value of the varied parameter.}
  \label{fig:numerical-profitability}
\end{figure}

Figure~\ref{fig:numerical-profitability}\subref{fig:numerical-profitability-mu1} shows that both barriers increase with $\mu_1$ over the displayed range.
The barrier of regime 1 rises from approximately $0.53$ to $0.90$, rapidly at lower drift levels and more slowly thereafter. 
The barrier of regime 2 rises from approximately $0.57$ to $1.18$, with substantial increases also at higher $\mu_1$. 
The response of $b_2^*$, despite unchanged $\mu_2$ and $\sigma_2$, shows how the prospect of entering a more profitable regime affects cash retention.

Figure~\ref{fig:numerical-profitability}\subref{fig:numerical-profitability-mu2} shows a different response to $\mu_2$. 
The barrier of regime 1 rises from approximately $0.70$ to $1.06$. 
The barrier of regime 2 first rises from approximately $0.73$ to $0.92$, peaks around $\mu_2=0.12$, and then declines to approximately $0.80$. 
The barriers cross near $\mu_2=0.12$; beyond this point, further increases lower $b_2^*$ and raise $b_1^*$.

Two competing considerations can help interpret these responses. Higher mean cash inflows make reserves easier to replenish, which can reduce the incentive to maintain a large cash reserve. 
At the same time, higher mean cash inflows improve future dividend opportunities, making it more attractive to retain cash that helps the firm continue operating. 
The rising barriers in Figure~\ref{fig:numerical-profitability}\subref{fig:numerical-profitability-mu1} are consistent with the
latter consideration. 
The declining part of $b_2^*$ in Figure~\ref{fig:numerical-profitability}\subref{fig:numerical-profitability-mu2} is consistent with the former consideration once profitability of regime 2 is sufficiently high. 
Meanwhile, the continued increase in $b_1^*$ suggests that preserving cash to reach the more profitable regime remains attractive.

\citet[Section~5.1.1, Table~2]{jiang2012optimal} also find that both barriers initially increase with drift and subsequently decrease in their numerical example. 
Thus, profitability need not have a monotone effect on dividend barriers. 
Here, both barriers increase with $\mu_1$, whereas increasing $\mu_2$ raises $b_1^*$ and produces a non-monotone response in $b_2^*$. 
These findings apply to the examined ranges; in particular, $b_1^*$ may decline at values of $\mu_1$ beyond those displayed.

\begin{figure}[htbp]
  \centering
  \captionsetup[subfigure]{font=small,labelfont=normalfont,textfont=normalfont,justification=centering}
  \begin{subfigure}[t]{0.485\linewidth}
    \centering
    \IfFileExists{Regime1_volatility.png}{%
      \includegraphics[width=0.98\linewidth]{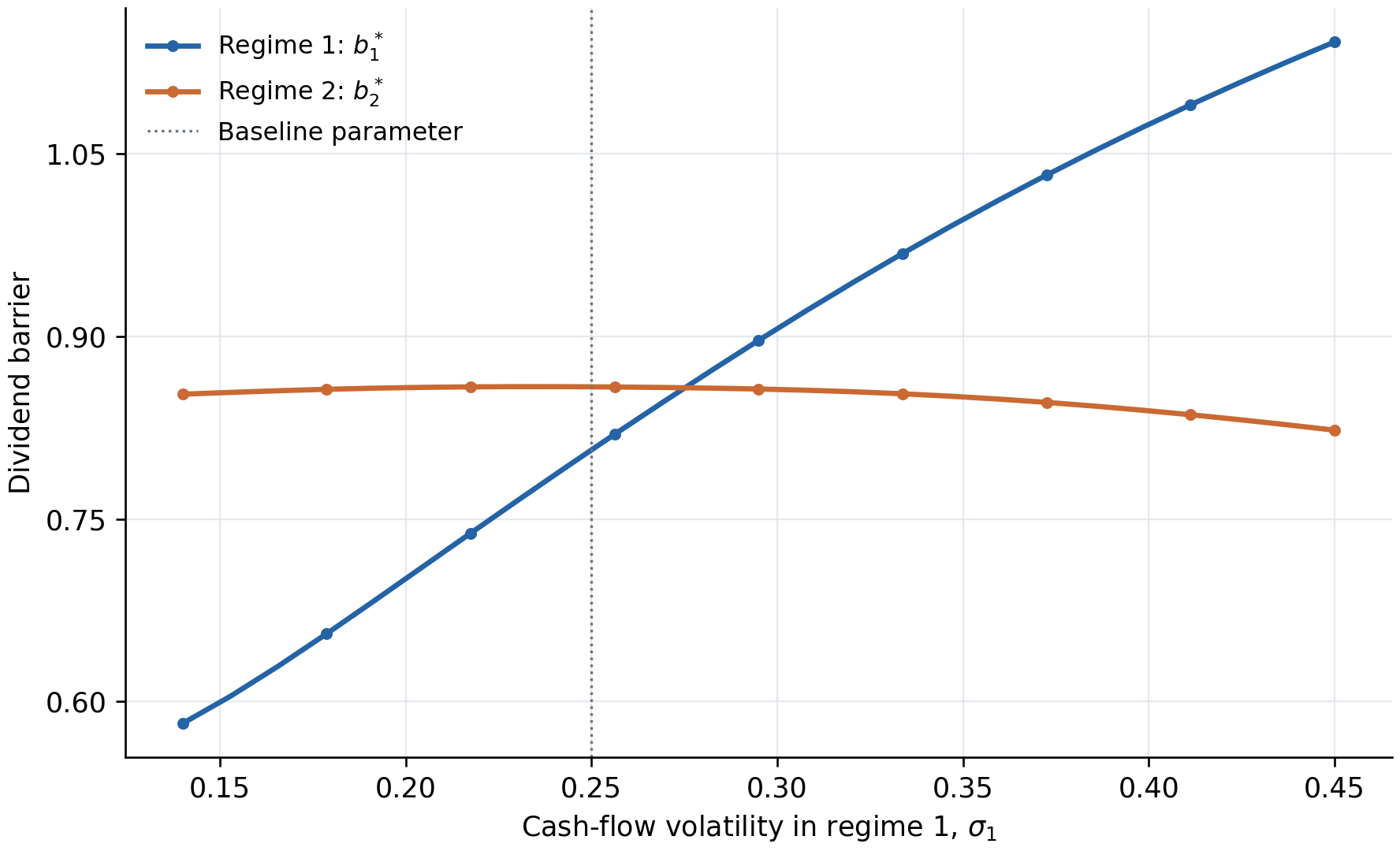}%
    }{%
      \fbox{\parbox[c][0.20\textheight][c]{0.92\linewidth}{%
        \centering
        Volatility and dividend barriers\\[0.5em]
        \texttt{\detokenize{Regime1_volatility.png}}%
      }}%
    }
    \caption{Regime 1: varying $\sigma_1$.}
    \label{fig:numerical-volatility-sigma1}
  \end{subfigure}\hfill
  \begin{subfigure}[t]{0.485\linewidth}
    \centering
    \IfFileExists{Regime2_volatility.png}{%
      \includegraphics[width=0.98\linewidth]{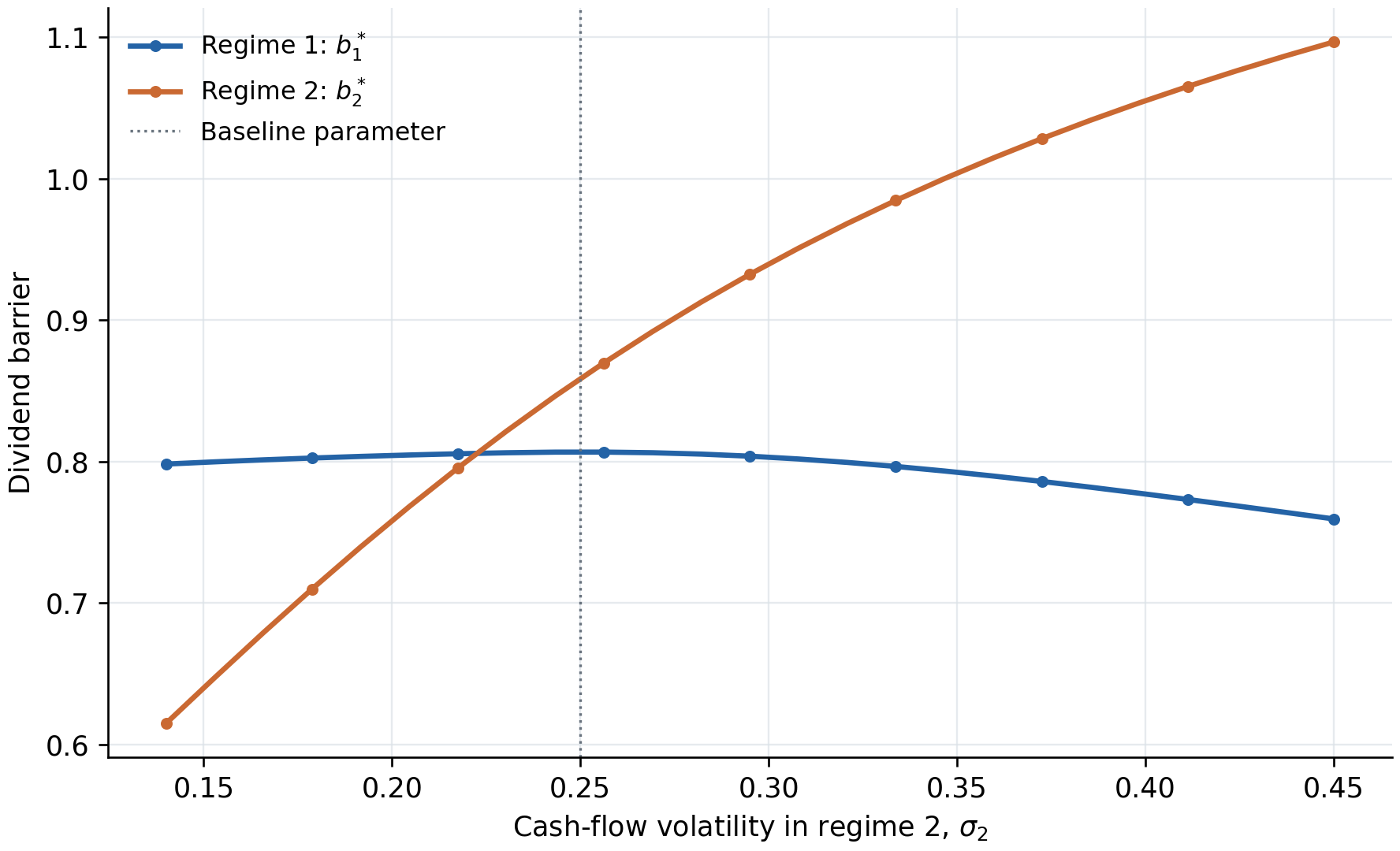}%
    }{%
      \fbox{\parbox[c][0.20\textheight][c]{0.92\linewidth}{%
        \centering
        Volatility and dividend barriers\\[0.5em]
        \texttt{\detokenize{Regime2_volatility.png}}%
      }}%
    }
    \caption{Regime 2: varying $\sigma_2$.}
    \label{fig:numerical-volatility-sigma2}
  \end{subfigure}
  \caption{Cash-flow volatility and equilibrium dividend barriers.
    Subfigure (a) varies $\sigma_1$, and Subfigure (b) varies $\sigma_2$.
    All remaining parameters are held at their baseline values.
    Vertical dotted lines mark the baseline value of the varied parameter.}
  \label{fig:numerical-volatility}
\end{figure}

Figure~\ref{fig:numerical-volatility}\subref{fig:numerical-volatility-sigma1} shows that increasing $\sigma_1$ raises the barrier in regime 1 substantially, from approximately $0.58$ to $1.14$. 
The response of $b_2^*$ is smaller: it rises slightly, is almost flat near the baseline, and then declines. 
The barriers cross at approximately $\sigma_1=0.28$.
Figure~\ref{fig:numerical-volatility}\subref{fig:numerical-volatility-sigma2} shows a similar pattern with the regimes reversed. 
Increasing $\sigma_2$ raises $b_2^*$ from approximately $0.62$ to $1.10$, while $b_1^*$ first rises slightly to about $0.81$ and then falls to approximately $0.76$. The barriers cross near $\sigma_2=0.22$.
In both experiments, the larger barrier adjustment occurs in the regime whose cash flows become more uncertain.

These increases are consistent with precautionary
cash retention: larger reserves can absorb adverse shocks and preserve future dividend opportunities. 
\citet{opler1999determinants} find that firms with riskier cash flows hold higher cash ratios, which provides evidence for this interpretation. 

\citet[Section~5.1.1, Table~2]{jiang2012optimal} report increases in both barriers at four volatility values, with a smaller change in the regime whose volatility is unchanged. 
These observations do not establish general monotonicity.
Our figures also show a smaller change in the regime whose volatility is unchanged, although its barrier is nonmonotone.

\bibliographystyle{apalike}
\bibliography{wpref}
\end{document}